\documentclass[11pt,twoside,letterpaper]{article} %% The same for {book}
\usepackage{times,fancyhdr}
\usepackage[dvips]{graphicx}

\def\figurename{Figure} % (Replace the colon that normally
\makeatletter
\renewcommand{\fnum@figure}[1]{\figurename~\thefigure.}
\makeatother

\def\tablename{Table} % (Replace the colon that normally
\makeatletter
\renewcommand{\fnum@table}[1]{\tablename~\thetable.}
\makeatother

\usepackage{mathrsfs}
\usepackage{amsmath}
\usepackage{amssymb}
\usepackage{amsfonts}
\usepackage{amsthm,amscd}
\usepackage{xcolor}
\usepackage{hyperref}

\newtheorem{theorem}{Theorem}[section]
\newtheorem{lem}[theorem]{Lemma}
\newtheorem{rmk}[theorem]{Remark}

\newtheorem{ass}[theorem]{Assumption}

\numberwithin{equation}{section}

\def\R{\mathbb R}

\def\cal{\mathcal}

\begin{document}
%\vskip 0.4in
\title{{\bfseries Global attractors and their  upper semicontinuity   for a structural damped wave equation with  supercritical nonlinearity on $\mathbb{R}^{N}$}
\thanks{ The work is supported by National Natural Science Foundation of China (No.11671045, No. 11671367).  * Corresponding author:  Pengyan Ding.  E-mail:  chen\_qionglei@iapcm.ac.cn,
 dingpengyan@yeah.net, yzjzzut@tom.com}}
\author{\bfseries Qionglei Chen$^a$, Pengyan Ding$^{*a}$, \ \  Zhijian  Yang$^{b}$\\
$^a$ Institute of Applied Physics and Computational Mathematics,
Beijing 100088, China\\
 $^b$ School of Mathematics and Statistics, Zhengzhou
University, No.100, Science Road, \\
 Zhengzhou  450001,   China }
\date{}
\maketitle \thispagestyle{empty} \setcounter{page}{1}

\begin{abstract}
The paper investigates the existence of global attractors and their upper semicontinuity  for a  structural damped wave equation on $\mathbb{R}^{N}: u_{tt}-\Delta u+(-\Delta)^\alpha u_{t}+u_{t}+u+g(u)=f(x)$, where $\alpha\in (1/2, 1)$  is called  a dissipative index. We propose  a new method based on  the harmonic analysis technique and the  commutator estimate to exploit the dissipative effect of the structural damping   $(-\Delta)^\alpha u_{t}$ and to overcome the essential difficulty:  ``both the unbounded domain $\mathbb{R}^N$ and the supercritical nonlinearity  cause that the Sobolev embedding loses its  compactness"; Meanwhile we show that there exists  a supercritical index  $p_\alpha\equiv\frac{N+4\alpha}{N-4\alpha}$ depending on  $\alpha$  such that   when the growth exponent $p$ of the nonlinearity $g(u)$ is up to the supercritical range: $1\leqslant p<p_\alpha$: (i)  the IVP of the equation is well-posed and its solution is of additionally global smoothness when $t>0$; (ii) the related solution semigroup possesses a global attractor $\mathcal{A}_\alpha$ in natural energy space for each $\alpha\in (1/2, 1)$; (iii) the family of global attractors   $\{\mathcal{A}_\alpha\}_{\alpha\in (1/2, 1) }$ is upper semicontinuous at each point $\alpha_0\in (1/2, 1)$.
\end{abstract}

\vspace{.08in} \noindent \textbf{Keywords}:   Structural damped wave equation; unbounded domain; well-posedness;   global attractor; upper semicontinuity.

\vspace{.08in} \noindent \textbf{2010 Mathematics Subject Classification:} Primary:  35B40, 35B41;
Secondary: 35B33,  35B65,  37L15.

\section{Introduction}

\hspace{5mm}  In this paper, we investigate the well-posedness, the existence of global attractors and their upper semicontinuity to
the following  structural damped wave equation on $\mathbb{R}^{N}$:
\begin{align}
&u_{tt}-\Delta u+(-\Delta)^\alpha u_{t}+u_{t}+u+g(u)=f(x), \ \ (x,t)\in
 \mathbb{R}^{N}\times  \mathbb{R}^+,\label{1.1}\\
&u(x,0)=u_{0}(x),\  u_{t}(x,0)=u_{1}(x), \ \  x\in\mathbb{R}^{N},\label{1.2}
\end{align}
where $\alpha\in (1/2, 1)$ is called a dissipative index.

It is well known that when $\Omega\subset \mathbb{R}^{N}$ is a bounded domain, the fractional damping
$(-\Delta)^\alpha u_t$ (which is especially called a structural damping as $\alpha\in (1/2, 1]$  (cf.   \cite{S-Chen, S-Chen1})) has a regularizing effect to  the solutions of the IBVP of Eq. \eqref{1.1}, i.e., it  makes them be of  additionally  global (when $\alpha\in (0,1)$) or partial (when $\alpha=1$) smoothness
when $t>0$. For example,  Chueshov \cite{Chueshov1,Chueshov} studied more general Kirchhoff wave model with  structural/strong nonlinear damping
\begin{align}\label{1.3}
u_{tt}-\phi(\|\nabla u\|^2)\Delta u+\sigma(\|\nabla u\|^2)(-\Delta)^\alpha u_t+g(u)=f(x),
\end{align}
with either $\alpha\in (1/2, 1)$ (cf. \cite{Chueshov1}) or $\alpha=1$ (cf. \cite{Chueshov})  on a bounded domain $\Omega$ with Dirichlet boundary condition.  When $ \alpha=1$, he found a supercritical exponent $p^{**}\equiv \frac{N+4}{(N-4)^+}$ ($a^+=max\{a,0\}$) and showed that when  the growth exponent  $p$  of the nonlinearity $g(u)$ is up to supercritical range: $1\leqslant p< p^{**}$, the IBVP of Eq. \eqref{1.3} is still well-posed, as well as its solutions possess additionally  partial regularity when $t>0$; Moreover  the related solution semigroup $S(t)$ has a finite-dimensional `partially strong' global attractor in natural phase space $X=(H^{1}_{0}\cap L^{p+1})(\Omega)
\times L^{2}(\Omega)$. Recently, Ding, Yang and Li \cite{Y-D-L1} removed this `partially strong' restriction for the global attractor.

When $\alpha\in (1/2, 1)$ and $\sigma(s)\equiv 1$, Eq. \eqref{1.3} becomes
\begin{align}\label{1.4}
u_{tt}-\phi(\|\nabla u\|^2)\Delta u+(-\Delta)^\alpha
u_t+g(u)=f(x),
\end{align}
Yang, Ding and Li  \cite{Y-D-L}  found a supercritical growth exponent $p_\alpha\equiv\frac{N+4\alpha}{(N-4\alpha)^+}$ depending on the dissipative index $\alpha$ and  showed   that  when  the growth exponent  $p$  of the nonlinearity $g(u)$ is up to supercritical range: $1\leqslant p< p_\alpha$, not only is the
IBVP of Eq. \eqref{1.4}   well-posed, but also its solutions possess    additionally  global regularity  when $t>0$ (rather than partial one as $\alpha=1$ case); Furthermore, the related solution semigroup $S(t)$ possesses a global and an exponential attractor in phase space $X=(H_0^1\cap L^{p+1})(\Omega)\times L^2(\Omega)$. These results improve   those in \cite{Chueshov1} to some extent. For the related researches on this topic on a bounded domain, one can see  \cite{Carvalho2, Y-D-L1, Y-L, Y-L-N} and references therein.

However for unbounded domain $\Omega$, such as $\Omega=\mathbb{R}^N$, the case becomes much more complex. One major difficulty is  the loss of the compactness  of the Sobolev   embedding, which is indispensable for obtaining  the existence of global attractor. Recently, Yang and Ding \cite{Y-D}  studied the well-posedness and longtime dynamics for the  strongly damped  Kirchhoff wave model on $\mathbb{R}^N$:
\begin{equation}
u_{tt}-M(\|\nabla u\|^2)\Delta u -\Delta u_t + u_t +g(x,u)  = f(x), \label{1.6}
\end{equation}
which includes the  strongly damped wave Eq. \eqref{1.1} (taking $\alpha=1$ there) as its special case by taking $M(s)\equiv1$. By using the tail cut-off method (cf. \cite{Wang}),  the authors obtained the existence of   global and exponential attractors  in phase space $\mathcal{H}= H^1(\R^N)\times L^2(\R^N)$ provided that the nonlinearity $g(x,u)$ is of at most critical growth $p^*\equiv\frac{N+2}{(N-2)^+}$. These  results extend the previous  ones on this topic in \cite{Yang}.

Roughly speaking, the tail cut-off method introduced in \cite{Y-D} is that the authors firstly split  $\mathbb{R}^N$ into the union of a ball with radius $R$ and its complement:  $ \mathbb{R}^N=B_R\cup B_R^C$, then establish the tail estimate on $B_R^C$ such that it is sufficiently small when $R$ suitably large; Secondly, by fixing $R$ they both establish the regularizing estimate and use the compactness of the Sobolev embedding on $B_R$ to get the asymptotical compactness of the related solution semigroup.
Unfortunately, the above technique  is not valid for the structural damping case:  $(-\Delta)^\alpha u_t$, with $\alpha\in (1/2, 1)$, because there exists an essential difference between the spectrum of the operator $-\Delta$  on a bounded domain $\Omega$ and that on  $\mathbb{R}^N$, which leads to that one neither uses the limiting process of  the approximating  solution sequence on the bounded balls $B_R$ to get the existence and the regularity of the   solutions of problem \eqref{1.1}-\eqref{1.2}, nor proceeds the tail estimate  as done in \cite{Y-D} because   integral  by parts fails in this case for the appearance of the structural damping.  Therefore, it needs to provide a new method to overcome these difficulties arising from both the structural damping and   unbounded domain $\mathbb{R}^N$.  The desired aim is to obtain the corresponding results as in a bounded domain as in \cite{Y-D-L}.

In the present paper, based on the Littlewood-Paley theory  we put forward  a double truncation method (rather than letting $B_R\rightarrow \mathbb{R}^N$ as $R\rightarrow\infty$ as done before) and use the  localization in frequency of harmonic analysis technique to make full use of   the dissipative effect  of the structural damping $(-\Delta)^\alpha u_{t}$ and  establish the well-psedness of problem \eqref{1.1}-\eqref{1.2} together with
the additionally global regularity of its solutions when $t>0$. Furthermore, we introduce the commutator estimate to conquer the difficulty in tail cut-off estimate arising from  the structural damping and the essential difficulty:  `both the unbounded domain $\mathbb{R}^N$ and  the supercritical nonlinearity  cause that the compactness of Sobolev embedding is no longer true '. What's more, we use the harmonic analysis technique to solve the robustness of the global attractors on the dissipative index.

The main contribution of this paper is that it  realizes the desired aim. In the concrete, there exists  a supercritical index  $p_\alpha\equiv\frac{N+4\alpha}{N-4\alpha}$ depending on  $\alpha$  such that   when the growth exponent $p$ of the nonlinearity $g(u)$ is up to the supercritical range: $1\leqslant p<p_\alpha$:

(i)  the IVP of Eq. \eqref{1.1}  is well-posed and its solution
is  of additionally global  smoothness when $t>0$;  (see Theorem \ref{T2.2})

(ii)  the related solution semigroup  possesses a global attractor $\mathcal{A}_\alpha$  in natural energy space for each $\alpha\in (1/2, 1)$; (see Theorem \ref{T3.7})

(iii) the family of global  attractors  $\{\mathcal{A}_\alpha\}_{\alpha\in (1/2, 1) }$ is upper semicontinuous at each point $\alpha_0\in (1/2, 1)$, i.e.,
for any neighborhood $U$ of $\mathcal{A}_{\alpha_0}, \mathcal{A}_\alpha\subset U$  when $|\alpha-\alpha_0|\ll1$. (see Theorem \ref{T4.2})

We mention  that there have been  several remedies to save the loss of compactness  of the Sobolev embedding  on the unbounded domain.  One of them is working in  weighted Sobolev spaces
(cf. \cite{Abergel, Babin, Karachalios, Mielke, Zelik}
and references therein). For example,  Savostianov \cite{S} studied the  infinite-energy solutions of the semilinear wave equation with fractional damping  in an unbounded domain of $\mathbb{R}^3$
 \begin{equation}
 u_{tt}-\Delta u + \lambda_0u+\gamma(I-\Delta)^\frac{1}{2} u_t + f(u) = g.
 \end{equation}
Under  the critical nonlinearity assumption:
 \[|f'(s)|\leqslant C(1 + |s|^4)  \  \ \hbox{and}\ \ f(s)s\geqslant -M,\]
he  proved the existence,  uniqueness and extra regularity   of the infinite-energy solutions, and established the existence of locally compact attractor of the corresponding solution semigroup.  For the related researches  on the longtime dynamics of a nonlinear evolution equation  in  weighted Sobolev spaces,  one can see \cite{D-Y, M-S, Zelik1} and references therein.

Another approach  developed for damped wave equation is working in the usual Sobolev spaces. For example, by using the property of finite speed of propagation, the Strichartz estimate and a suitable semigroup decomposition, Feireisl \cite{Feireisl1, Feireisl2}  established the  existence  of   global attractor for the  damped semilinear  wave equation on $\mathbb{R}^3$:
 \begin{equation}\label{1.7}
u_{tt}-\Delta u + d(x)u_t+au+ f(x,u) = g(x)
 \end{equation}
provided that $|f_s(x,s)|\leqslant C(1+|s|^q), q<4$.

By combining the  decomposition of the solution semigroup with the  suitable cut-off functional,  Belleri and Pata \cite{Belleri}
proved the existence  of  global attractor for the  strongly damped semilinear wave equation on $\mathbb{R}^3$:
 \begin{eqnarray}
u_{tt}-\Delta u_{t}-\Delta u+g(x,u)+\phi(x)u_{t}=f(x) \label{1.8}
\end{eqnarray}
provided that the nonlinearity $g(x,u)$ is of subcritical growth $3$. Then in critical nonlinearity case,  namely, the growth rate of $g(x,u)$ can  reach to  $5$, Conti, Pata and Squassina \cite{Conti}  obtained the existence of global attractor for Eq. \eqref{1.8} (replacing  $\phi(x)u_{t}$ there by more general $\phi(x, u_{t})$).

Obviously, in all the above mentioned researches, the growth exponent of the nonlinearity reaches at most to critical.  To the best of the authors' knowledge, this paper is the first one to establish the existence of global attractors and   their robustness  on  the dissipative index in supercritical nonlinearity case on $\mathbb{R}^N$.

The paper is arranged  as follows.  In Section \ref{Sec2}, we discuss the  well-posedness of   problem \eqref{1.1}-\eqref{1.2} and the additionally global smoothness of its solutions when $t>0$. In Section \ref{Sec3},  we study the existence of global attractor for each $\alpha\in (1/2, 1)$. In Section \ref{Sec4}, we investigate
upper semicontinuity of the family of global attractors $\{\mathcal{A}_\alpha\}_{\alpha\in (1/2, 1)}$.

\section{\bf Well-posedness }\label{Sec2}

We begin with  the following abbreviations:
\begin{eqnarray*}&&
L^{p}=L^{p}(\mathbb{R}^{N}),\ \ H^{s,p}=H^{s,p}(\mathbb{R}^{N}), \ \ \dot{H}^{s,p}=\dot{H}^{s,p}(\mathbb{R}^{N}),\ \  \int=\int_{\R^N},\ \  \|\cdot\|_{p}=\|\cdot\|_{L^{p}},\ \   \|\cdot\|=\|\cdot\|_{2},
\end{eqnarray*}
with $s\in \mathbb{R}, p>1$, where  $H^{s,p}$ is the Bessel potential space equipped with the norm
\[\|f\|_{H^{s,p}}=\|(I-\Delta)^{\frac{s}{2}}f\|_p
=\|\mathcal{F}^{-1}[(1+|\xi|^2)^{\frac{s}{2}}\mathcal{F}f]\|_p,\]
and where $I$ denotes the identity operator,   $\hat{f}=\mathcal{F}(f)$ and $\check{f}= \mathcal{F}^{-1}(f)$ are  the Fourier transformation and the Fourier inverse transformation of $f$, respectively.   And  $\dot{H}^{s,p}$ is the Riesz potential space  equipped with the semi-norm
\[\|f\|_{\dot{H}^{s,p}}=\|(-\Delta)^{\frac{s}{2}}f\|_p
=\|\mathcal{F}^{-1}[|\xi|^s\mathcal{F}f]\|_p,\]
where the operators $(I-\Delta)^{\frac{s}{2}}$ and $(-\Delta)^{\frac{s}{2}}$ are called Bessel potential and Riesz potential, respectively. The notation $(\cdot,\cdot)$ for $L^{2}$-inner product will also be used for the notation of
duality pairing between the dual spaces. We use the same letter $C$ to denote different positive constants, and use $C(\cdot,\cdot)$ to denote positive constants depending on the quantities appearing in  parenthesis. The sign $H_{1}\hookrightarrow H_{2}$ denotes that the functional space $H_{1}$ continuously embeds into $H_{2}$ and
$ H_{1}\hookrightarrow\hookrightarrow H_{2}$ denotes that $H_{1}$ compactly embeds into $H_{2}$.

Define the phase spaces
\begin{equation}
\mathcal{H}_\alpha=
 H^{1+\alpha}\times H^\alpha, \ \  \mathcal{H}=(H^{1}\cap L^{p+1})\times L^2, \ \  \mathcal{H}_{-\alpha}
   =H^\alpha\times H^{-\alpha},\label{e2.1}
\end{equation}
which are equipped with  usual graph norms, for instance,
\[\|(u, v)\|_{\mathcal{H}}^2=\|u\|_{H^1}^2+d_0\|u\|^2_{p+1}+\|v\|^2, \ \ \forall (u, v)\in \mathcal{H},\]
where
\begin{align*}
d_0=\begin{cases} 0, &\ \ \hbox{if} \ \ 1\leqslant p\leqslant p^*\equiv\frac{N+2}{(N-2)^+},\\
1, &\ \ \hbox{if} \ \ p^*< p< p_\alpha\equiv\frac{N+4\alpha}{(N-4\alpha)^+}.
\end{cases}
\end{align*}
Obviously,
 \[\mathcal{H}_\alpha\hookrightarrow \cal H\hookrightarrow \mathcal{H}_{-\alpha}\]
for each $\alpha\in (1/ 2, 1)$.

Define the smooth radial  function
\[\hat{\varphi}(\xi)=
\begin{cases}
1, & \text{if} \ \ |\xi|\leqslant 1, \\
0, & \text{if} \ \ |\xi|\geqslant 2,
\end{cases}, \ \
\int_{\mathbb{R}^N} \hat{\varphi}(\xi) d\xi=1,\ \ \hat{\varphi}_l(\xi)=\hat{\varphi}(\frac{\xi}{2^l}),\]
where $\hat{\varphi}=\mathcal{F}(\varphi)$. Define the operator $S_l$:
\begin{equation*}
(S_lu)(x)=\mathcal{F}^{-1}(\hat{\varphi}_l\hat{u})(x)
=(\varphi_l\ast u)(x),
\end{equation*}
where  $\varphi_l(x)=2^{l N}\varphi(2^{l}x) (\in L^1)$.

\begin{lem} \cite{Chemin}\label{L2.2}
Let  $\mathrm{B}$ be a ball. Then for any non-negative integer $k$, any couple of real  $(a, b)$, with $b\geqslant a\geqslant 1$, and any function $u\in L^a$, there exists a positive number $C$ such that
\begin{align}\label{2.14'}&
\text{supp}\ \hat{u} \subset \lambda \mathrm{B}\Rightarrow
\sup_{|\alpha|=k}\|\partial^\alpha u\|_{L^b}\leqslant C^{k+1}\lambda^{k+N(\frac{1}a-\frac{1}{b})}\|u\|_{L^a}.
\end{align}
\end{lem}

\begin{lem} (Properties of the operator $S_l$)\label{L2.4}
\ (i) For any $m\in \mathbb{R}$, if $u\in H^m$, then $S_lu\in H^n, \forall n\geqslant m$, and
\begin{align*}
\|S_lu\|_{H^n}\leqslant C(l, m, n)\|S_lu\|_{H^{m}},\label{2.10.}
\end{align*}
especially,
\begin{equation}
\|S_lu\|_{H^{m}}\leqslant \|u\|_{H^{m}}  \ \ \hbox{and}\ \ \lim_{l\rightarrow\infty}\|S_lu-u\|_{H^m}=0.\label{e2.9.}
\end{equation}

(ii) If $u\in L^q$  with $1 \leqslant q < \infty$, then $S_l u\in L^p, \forall p\geqslant q$, and
\begin{align*}
\|S_lu\|_{p}\leqslant C(l,p,q)\|S_lu\|_{q},
\end{align*}
 especially,
\begin{align*}
 \|S_l u \|_q \leqslant \varphi_0 \| u \|_q  \ \ \hbox{and}\ \ \lim_{l\rightarrow\infty}\|S_l u - u\|_q=0,
\end{align*}
where $\varphi_0 \equiv\|\varphi\|_1$.

(iii)
\[(S_lu,v)=(u,S_lv),\ \ \forall u,v \in L^2.\]
\end{lem}

\begin{proof}
(i) For any  $n\geqslant m$,
 \begin{align*}&
\|S_lu\|_{H^n}
=\|(1+|\xi|^2)^{\frac{n-m}2}
(1+|\xi|^2)^{\frac{m}2}(\hat{\varphi}_l\hat{u})\|\\
\leqslant& C(l, m, n)\|
(1+|\xi|^2)^{\frac{m}2}(\hat{\varphi}_l\hat{u})\|
=C(l, m, n)\|S_lu\|_{H^{m}}.
\end{align*}
Obviously, $C(l, m, n)=1$ when $n=m$. By the Plancherel theorem,
 \begin{align*}
 \|S_lu-u\|_{H^m}^2&=\int(1+|\xi|^2)^m(\hat{\varphi}_l-1)^2|\hat{u}|^2d\xi\\
 &
 \leqslant C\int_{|\xi|>2^l}(1+|\xi|^2)^m|\hat{u}|^2d\xi\rightarrow 0\ \ as \ \ l\rightarrow \infty.
 \end{align*}

(ii) By formula   \eqref{2.14'} (taking that $B$ is a unit ball, $k=0, \lambda=2^{l+1}, a=q, b=p$ there), we have
\[ \|S_l u \|_p \leqslant C(l,p,q) \| u \|_q.\]
Taking account  of $ \varphi, \varphi_l\in L^1$,  we obtain
\[\| S_l u \|_q = \| \varphi_l *u\|_q\leqslant \| \varphi_l \|_1 \| u \|_q=\| \varphi\|_1 \|u \|_q\leqslant \varphi_0\| u\|_q.\]
By the density of $C^\infty_0(\mathbb{R}^N)$ in $L^q$ we have that for any $\epsilon >0$, there exists a function $g\in C^\infty_0(\mathbb{R}^N)$ satisfying
\[\|u-g \|_q < \frac{\epsilon}{2(\varphi_0+1)}. \]
Since $H^m \hookrightarrow L^q$ for $m$ suitably large and  \eqref{e2.9.}, there must exist a   $L>0$ such that
\[\| S_l g- g\|_q\leqslant C\| S_l g- g\|_{H^m}<  \frac{\epsilon}{2} \ \ as\ \  l> L.\]
 Therefore,
\begin{align*}
\|S_lu - u\|_q
\leqslant&\| S_l u - S_l g \|_q+\| S_l g-g\|_q+\|g-u\|_q \\
\leqslant &(\varphi_0+1)\|u - g\|_q + \| S_l g- g\|_q < \epsilon, \ \ as\ \  l> L.
\end{align*}

(iii) For any $u, v\in L^2$,
\[(S_lu,v)=(\mathcal{F}^{-1}(\hat{\varphi}_l\hat{u}),v)
=(\hat{\varphi}_l\hat{u},\hat{v})=(\hat{u},\hat{\varphi}_l\hat{v})
=(u,S_lv).\]
\end{proof}

Let the phase space  $X_\alpha=H^{2\alpha+1}\times H^{2\alpha}$, with $\alpha\in (1/ 2, 1)$. By Lemma \ref{L2.4},
\begin{align}\label{e2.9*}
 &\|(S_lu_0, S_lu_1)\|_{X_\alpha}\leqslant C_0(l)\|(u_0, u_1)\|_{\cal H},\ \  \ \|(S_lu_0, S_lu_1)\|_{ \mathcal{H}}\leqslant (\varphi_0+1)\|(u_0, u_1)\|_\mathcal{H},\nonumber\\
&   \|S_lf\|\leqslant \|f\|,\ \ \lim_{l\rightarrow\infty}\|(S_lu_0, S_lu_1, S_lf)-(u_0, u_1, f)\|_{ \mathcal{H}\times L^2}=0,
\end{align}
where and in the following $ C_0(l)$ denotes a positive constant  depending  on $l$.

\begin{lem} \cite{Simon}\label{L2.1}
Let $X, B$ and $Y$ be three Banach spaces, $X \hookrightarrow\hookrightarrow B\hookrightarrow Y,$
\begin{eqnarray*}&&
W =\{u\in L^p(0,T;X)|  u_t\in L^1(0,T;Y)\}, \ \hbox{with}\
1\leqslant p<\infty,\\
&& W_1=\{u\in L^\infty(0,T;X)| u_t\in L^r(0,T;Y)\},\ \hbox{with}\
r>1.
\end{eqnarray*}
Then $ W \hookrightarrow\hookrightarrow
L^p(0,T;B), W_1\hookrightarrow\hookrightarrow C([0,T];B)$.
\end{lem}

\begin{ass} \label{A2.3}
(i)  $g \in C^1(\mathbb{R})$,  with $g(0)=0$. Either
\begin{align*}&
 |g'(s)|\leqslant  C(1+|s|^{p-1}) \ \ \text{and }  g'(s)>-C_0 \ \ \text{for some }\ 1\leqslant p\leqslant p^*,
\end{align*}
or else
\begin{align*}
 C_1|s|^{p-1}-C_0\leqslant g'(s)\leqslant  C(1+|s|^{p-1})\ \ \text{for some }  \  p^*< p< p_\alpha\equiv\frac{N+4\alpha}{(N-4\alpha)^+},
\end{align*}
where $C_1> 0, 0< C_0< 1$ and $a^{+}=\max\{0, a\}$.

(ii) $f\in L^2, \xi_{u}(0)=(u_{0},u_{1})\in \mathcal{H},\|(u_{0},u_{1})\|_{\mathcal{H}}\leqslant R_0$ for some positive constant $R_0$.
\end{ass}

\begin{rmk}\label{R2.4}
It follows from Assumption \ref{A2.3} (i)  that
\begin{align*}&
 |g(s_1)-g(s_2)|\leqslant C(1+|s_1|^{p-1}+|s_2|^{p-1})|s_1-s_2|, \ \ |g(s)|\leqslant C(|s|+|s|^{p}),\\
&g(s)s\geqslant\frac{d_0C_1}{p}|s|^{p+1}-C_0s^2, \ \
G(s)\geqslant\frac{d_0C_1}{p(p+1)}|s|^{p+1}-\frac{C_0}{2}s^2,\\%\label{a2.6}\\
&g(s)s-G(s)+\frac{C_0}{2}s^2\geqslant0
\end{align*}
for   $1\leqslant   p< p_\alpha$  (which means $H^{2\alpha}\hookrightarrow L^{p+1}$), where and in the following $G(s)=\int^s_0g(\tau)d\tau$.
\end{rmk}

\begin{theorem} \label{T2.2}
Let Assumption \ref{A2.3} be valid.  Then problem  \eqref{1.1}-\eqref{1.2}  admits a unique solution $u (=u^\alpha)$ for each $\alpha\in (1/2, 1)$, with
$(u, u_t) \in L^\infty(\mathbb{R}^{+}; \mathcal{H})\cap C_w(\mathbb{R}^{+}; \mathcal{H})$, and the solution possesses the following properties:\\
\noindent(i) (Dissipativity)
\begin{align}
&\|u(t)\|^{2}_{H^{1}}+\|u(t)\|^{2}_{p+1}+\|u_{t}(t)\|^{2}\leqslant
C(R_0, \|f\|)e^{-\kappa t}+ C\|f\|^2,\ \ t>0, \nonumber\\
&\int^{\infty}_{0}\|u_{t}(\tau)\|^{2}_{H^\alpha}d\tau\leqslant C(R_0,\|f\|).\label{2.1}
\end{align}
\noindent(ii) (Additional regularity  as $t> 0$)  For any $a>0, T> a> 0$,
\[(u_{t}, u_{tt})\in L^{\infty}(a,T;H^\alpha\times H^{-\alpha})\cap  L^2(a,T; H^1\times L^2), \]
 \begin{align}
&\|u_{t}(t)\|_{H^\alpha}^{2}+\|u_{tt}(t)\|^2_{H^{-\alpha}}+\int^{t+1}_{t}
(d_0\int|u(\tau)|^{p-1}|u_t(\tau)|^2dx
+\|u_t(\tau)\|_{H^1}^2+\|u_{tt}(\tau)\|^2)d\tau\nonumber\\
&
\leqslant(1+\frac{1}{t^2})C(R_0, \|f\|),\ \ t>0. \label{2.3}
\end{align}
Moreover, $u\in L^{\infty}(a, T; H^{1+\alpha})\cap L^{2}(t, t+1; H^{2})$ and
\begin{align}&
&\|u(t)\|_{H^{1+\alpha}}^{2}+\int^{t+1}_t\|u(\tau)\|^2_{H^{2}}d\tau
\leqslant \left( 1+\frac{1}{(t-\frac{a}{2})^\frac{1}{1-\alpha}}\right)C(\alpha)C(R_0,\|f\|),\ \  t>\frac{a}{2},\label{2.4}
\end{align}
where  $C(\alpha)= (\frac{C\alpha}{1-\alpha})^\frac{\alpha}{1-\alpha}$.\\
\noindent (iii) (Lipshitz stability and quasi-stability in weaker space $\cal{H}_{-\alpha}$)
\begin{align}
\|(z, z_t)(t)\|^{2}_{\cal{H}_{-\alpha}}\leqslant C(T)\|(z, z_t)(0)\|^{2}_{\cal{H}_{-\alpha}}, \ \ t\in [0,T],  \label{2.5}
\end{align}
and
\begin{align}
\|(z, z_t)(t)\|^{2}_{\cal{H}_{-\alpha}}\leqslant Ce^{-\kappa t}\|(z, z_t)(0)\|^{2}_{\cal{H}_{-\alpha}}
+\int^t_0e^{-\kappa (t-\tau)}\|(z, z_t)(\tau)\|^{2}_{L^2\times H^{-2\alpha}}d\tau, \ \ t> 0,  \label{2.6}
\end{align}
where and in the following $\kappa$ denotes a small positive constant, and where $ z=u-v, u, v$ are two  solutions of problem \eqref{1.1}-\eqref{1.2} corresponding to initial data $(u_{0},u_{1})$ and $(v_{0},v_{1})$
 with $\|(u_{0},u_{1})\|_{\mathcal{H}}+\|(v_{0},v_{1})\|_{\mathcal{H}}\leqslant R_0$, respectively.
\end{theorem}

\begin{proof}  We first consider  the following auxiliary problem on $\mathbb{R}^N$:
\begin{align}\label{2.9''}
\begin{cases} u^l_{tt}-\Delta u^l+(-\Delta)^\alpha u^l_{t}+u^l_{t}+u^l+S_l(g(S_lu^l))=f^l,  \\
u^l(0)=u^l_{0}, \ \ u^l_{t}(0)=u^l_{1},
 \end{cases}
\end{align}
where $(u^l_{0}, u^l_{1})=(S_l u_0, S_l u_1), f^l=S_l f$. Rewrite this problem as the equivalent form on $\mathbb{R}^N$:
\begin{align}
\begin{cases}
\frac{d}{dt}U^l(t)=BU^l(t)+F(u^l),\\
U^l(0)=U^l_0,&
\end{cases} \label{2.9}
\end{align}
where
\begin{align*}
U^l=\begin{pmatrix}
u^l  \\
 u_t^l
\end{pmatrix},\ \
B=\begin{pmatrix}
0 & I \\
\Delta-I & -(-\Delta)^\alpha-I
\end{pmatrix},\ \
F(u^l)=\begin{pmatrix}
0\\ -S_l(g(S_lu^l))+f^l
\end{pmatrix},\ \
U^l_0=\begin{pmatrix}
u^l_{0}\\ u^l_{1}
\end{pmatrix}.
\end{align*}

We show that  problem \eqref{2.9} admits a unique global solution $U^l$ for each $l>0$. By Theorem 2.1 in \cite{G-G-H}, we know that the following linear problem on $\mathbb{R}^N$:
\begin{align*}
\begin{cases}
\frac{d}{dt}U^l(t)=BU^l(t),\\
U^l(0)=U^l_0&
\end{cases}
\end{align*}
possesses a unique solution $U^l\in  C_b (\mathbb{R}^+; X_\alpha)$ for each $U^l_0\in X_\alpha$.   Define the solution operator
 \[\Sigma(t): X_\alpha\rightarrow X_\alpha, \ \ \Sigma(t)U^l_0=U^l(t),\ \ \forall t\in \mathbb{R}^+.\]
Then $\{\Sigma(t)\}_{t\geqslant 0}$ constitutes a semigroup on $X_\alpha$  and a simple calculation shows that
\begin{equation}
\|\Sigma(t)U^l_0\|_{X_\alpha}\leqslant  \|U^l_0\|_{X\alpha},\ \  \forall U^l_0\in  X_\alpha, \ \ i.e., \ \ \|\Sigma(t)\|_{\mathcal{L}(X_\alpha)}\leqslant 1,\ \ \forall t>0. \label{E2.15}
\end{equation}
Let the space
\begin{align*}Z =
C([0, T]; X_\alpha)\ \ \hbox{equipped with the norm}\ \ \|U\|_Z= \max_{0\leqslant  t\leqslant  T}\|U(t)\|_{X_\alpha},
\end{align*}
and
\begin{align*}
X_{R, T} =\{ U\in Z\mid \|U\|_{Z}\leqslant 2R, U(0)=U^l_0\}, \ \ \hbox{with}\ \ R> C_0(l)R_0+\|f\|,
\end{align*}
equipped with the metric
$ d(U, V) = \|U-V\|_Z$. Obviously, $X_{R, T}$  is a complete metric space. For problem \eqref{2.9} we define the operator $\mathbb{T}: X_{R, T} \rightarrow Z$, for any $ U^l\in X_{R, T}$,
\begin{equation}
\mathbb{T}U^l(t)= \Sigma(t) U^l_0+\int^t_0\Sigma(t-\tau)F(u^l(\tau))d\tau, \ \    t\in (0,T].\label{2.15a}
\end{equation}

Now, we show that $\mathbb{T}$ is a contraction  mapping from $X_{R,T}$ to itself for $T$ suitably small.

(i) By  Lemma \ref{L2.4}, formulas  \eqref{E2.15}, \eqref{2.15a}  and the Sobolev embedding $H^{2\alpha+1}\hookrightarrow H^{2\alpha }\hookrightarrow L^{p+1}$  for $1\leqslant p< p_\alpha$,  we have
 \begin{align}
  \|\mathbb{T}U^l(t)\|_{X_\alpha}
 \leqslant & \|U^l_0\|_{X_\alpha}+T\|S_l(g(S_lu^l))-f^l\|_{C([0, T];H^{2\alpha})}\nonumber\\
 \leqslant & \|U^l_0\|_{X_\alpha}+C(l)T\left(\|g(S_lu^l)\|_{C([0, T];L^2)}+\|f^l\|\right)\nonumber\\
 \leqslant & \|U^l_0\|_{X_\alpha}+C(l)T\left(\|S_lu^l\|_{C([0, T];L^2)}+\|S_lu^l\|^p_{C([0, T];L^{2p})}+\|f^l\|\right)\nonumber\\
  \leqslant & \|U^l_0\|_{X_\alpha}+C(l)T\left(\|S_lu^l\|_{C([0, T];L^2)}+\|S_lu^l\|^p_{C([0, T];L^{p+1})}+\|f^l\|\right)\nonumber\\
 \leqslant & \|U^l_0\|_{X_\alpha}+C(l)T\left(\|u^l\|^p_{C([0, T];H^{2\alpha+1})}+\|f^l\|\right)\nonumber\\
 \leqslant & R+C_1(l)T(2R)^p \leqslant 2R,\ \ t\in [0,T] \label{2.15.}
\end{align}
for  $0<T<(2C_1(l)(2R)^{p-1})^{-1}$,  which means that  $\mathbb{T}: X_{R, T} \rightarrow X_{R, T}$.

(ii) Similar to the proof of \eqref{2.15.} and by the H\"{o}lder inequality, we have
\begin{align*}
 \|\mathbb{T}U^l-\mathbb{T}V^l\|_{Z}
 \leqslant & CT
 \|S_l(g(S_lu^l))-S_l(g(S_lv^l))\|_{C([0,T]; H^{2\alpha})}\\
  \leqslant & C(l)T\|g(S_lu^l)-g(S_lv^l)\|_{C([0,T];L^2)}\\
 \leqslant & C(l)T\big(\|S_lu^l-S_lv^l\|_{C([0,T];L^2)}\\
 &+ (\|S_lu^l\|^{p-1}_{C([0,T]; L^{2(p+1)})}
 +\|S_lv^l\|^{p-1}_{C([0,T]; L^{2(p+1)})})
 \|S_lu^l-S_lv^l\|_{C([0,T]; L^{p+1})}\big)\\
   \leqslant & C(l)T\big(\|u^l-v^l\|_{C([0,T];L^2)}\\
   &+ (\|S_lu^l\|^{p-1}_{C([0,T]; L^{p+1})}
   +\|S_lv^l\|^{p-1}_{C([0,T]; L^{p+1})})
   \|u^l-v^l\|_{C([0,T]; L^{p+1})}\big)\\
 \leqslant & C_2(l)T(2R)^{p-1}
\|U^l-V^l\|_{Z}
\leqslant\frac{1}{2}\|U^l-V^l\|_{Z}
 \end{align*}
for $0<T< (2C_2(l)(2R)^{p-1})^{-1}$. Taking
\[T: 0<T< \min\left\{\frac{1}{2C_1(l)(2R)^{p-1}}, \frac{1}{2C_2(l)(2R)^{p-1}}\right\},\]
then  $\mathbb{T}$ is a contraction  mapping from $X_{R,T}$ to itself. By the  Banach fixed point theorem, the mapping $\mathbb{T}$ has a unique fixed point, i.e.,   problem \eqref{2.9} has a   unique solution $U^l\in C([0,T];X_\alpha)$.   Let $[0, T_l)$ be the maximal interval of existence of  the solution $U^l$, i.e.,    $U^l\in C([0, T_l); X_\alpha)$. By the standard arguments we know that  if
\begin{equation}
\sup_{0\leqslant t< T_l}\|U^l(t)\|_{X_\alpha}< \infty,\label{e2.19}
\end{equation}
then $T_l=\infty$.

In order to prove $T_l =\infty$, we first give some a prior estimates for $U^l$. Using the multiplier $u_t^l+\epsilon u^l$ in Eq. \eqref{2.9''}, we get
\begin{equation}
\frac{d}{dt}H(\xi_{u^l}(t))+\Phi(\xi_{u^l}(t))=0,\label{2.12}
\end{equation}
where $\xi_{u^l}=(u^l, u_t^l)$,
\begin{align*}
H(\xi_{u^l})=&\frac{1}{2}\Big(\| u_{t}^l\|^2+\|u^l\|^2_{\dot{H}^1}+\|u^l\|^2+2\int G(S_lu^l)dx
-2(f^l,u^l)\Big)\\
&+\epsilon\Big((u^l,u_{t}^l)+\frac{1}{2}(\|u^l\|^{2}_{\dot{H}^{\alpha}}+\|u^l\|^{2})\Big),\nonumber\\
\Phi(\xi_{u^l})=&
\|u_{t}^l\|^{2}_{\dot{H}^\alpha}+(1-\epsilon)\|u_{t}^l\|^{2}
+\epsilon \Big(\|u^l\|^2_{\dot{H}^1}+\|u^l\|^2
+(g(S_lu^l), S_lu^l)-(f^l, u^l)\Big).
\end{align*}
It follows from Remark \ref{R2.4} that
\begin{align}\label{2.13}&
H(\xi_{u^l})\geqslant
\kappa(\| u^l_{t}\|^2+\|u^l\|^{2}_{H^{1}}
+d_0\|S_lu^l\|^{p+1}_{p+1})-C\|f^l\|^2,\nonumber\\
&\Phi(\xi_{u^l})-\epsilon H(\xi_{u^l})\geqslant 0
\end{align}
for $\epsilon>0$  suitably small,  and by  \eqref{e2.9*},
\begin{align*}
H(\xi_{u^l}(0))\leqslant&  C\Big(\| u_1^l\|^2+\|u_0^l\|^{2}_{H^{1}}+\|S_lu_0^l\|^{p+1}_{p+1}+\|f^l\|^2\Big)
\leqslant C(R_0)+C\|f\|^2.
\end{align*}
Inserting  \eqref{2.13} into \eqref{2.12} gives
\begin{align}\label{2.15}
&\frac{d}{dt}H(\xi_{u^l}(t))+\epsilon H(\xi_{u^l}(t))\leqslant C\|f\|^2,\nonumber\\
&
\|u^l(t)\|_{H^{1}}^{2}+d_0\|S_lu^l(t)\|^{p+1}_{p+1}+\|u_t^l(t)\|^{2}
\leqslant C(R_0,\|f\|)e^{-\kappa t}
+ C\|f\|^2, \ \ t>0.
\end{align}
Letting $\epsilon=0$ in \eqref{2.12}, integrating the resulting expression over $(0, t)$ and using \eqref{2.15},  we have
\begin{eqnarray}
\int^{t}_{0}\|u_{t}^l(\tau)\|^{2}_{H^\alpha}d\tau\leqslant C(R_0,\|f\|),\ \ t>0.\label{2.16}
\end{eqnarray}
The combination of \eqref{2.15} and \eqref{2.16} means that formula  \eqref{2.1} uniformly holds for $U^l(t)$, and the control constants in the right hand side are independent of $l$.

Using the multiplier $(-\Delta)^{2\alpha}u_t^l+\epsilon (-\Delta)^{2\alpha}u^l$ in Eq.  \eqref{2.9''} yields
\begin{equation}
\frac{d}{dt}H_1(\xi_{u^l})+\Phi_1(\xi_{u^l})=(f^l-S_l(g(S_lu^l)), (-\Delta)^{2\alpha}u_t^l+\epsilon (-\Delta)^{2\alpha}u^l),\label{2.12*}
\end{equation}
where
\begin{align}
H_1(\xi_{u^l})=&\frac{1}{2}\Big(\|u_{t}^l\|^2_{\dot{H}^{2\alpha}}+\|u^l\|^2_{\dot{H}^{2\alpha+1}}
+\|u^l\|^2_{\dot{H}^{2\alpha}}\Big)\nonumber\\
&+\epsilon\Big(((-\Delta)^{2\alpha}u^l, u_{t}^l)+\frac{1}{2}(\|u^l\|^{2}_{\dot{H}^{3\alpha}}
+\|u^l\|^{2}_{\dot{H}^{2\alpha}})\Big)\nonumber\\
&\sim\|u_{t}^l\|^2_{\dot{H}^{2\alpha}}
+\|u^l\|^2_{\dot{H}^{2\alpha+1}}
+\|u^l\|^2_{\dot{H}^{2\alpha}},\label{2.8*}\\
\Phi_1(\xi_{u^l})=&
\|u_{t}^l\|^{2}_{\dot{H}^{3\alpha}}
+(1-\epsilon)\|u_{t}^l\|^{2}_{\dot{H}^{2\alpha}}
+\epsilon \Big(\|u^l\|^2_{\dot{H}^{2\alpha+1}}+\|u^l\|^2_{\dot{H}^{2\alpha}}
\Big)\nonumber
\end{align}
for $\epsilon$  suitably small. By Lemmas \ref{L2.2} and  \ref{L2.4},
\begin{align}
\left|\left(S_l(g(S_lu^l)), (-\Delta)^{2\alpha}u_t^l\right)\right|
\leqslant&\|u_t^l\|_{\dot{H}^{3\alpha}}
\|S_l(g(S_lu^l))\|_{\dot{H}^{\alpha}}\nonumber\\
\leqslant&C(l)\|u_t^l\|_{\dot{H}^{3\alpha}}
\|g(S_lu^l)\|\nonumber\\
\leqslant&C(l)\|u_t^l\|_{\dot{H}^{3\alpha}}(\|S_lu^l\|
+\|S_lu^l\|_{2p}^{p})\nonumber\\
\leqslant&\frac{1}{4}\|u_t^l\|_{\dot{H}^{3\alpha}}^2
+C(l)(\|S_lu^l\|^2+\|S_lu^l\|_{p+1}^{2p}).\label{2.13*}
\end{align}
By the similar argument as to the  estimate \eqref{2.13*}, we have
\begin{align}\label{2.15*}&
\left|\left(S_l(g(S_lu^l)), (-\Delta)^{2\alpha}u^l\right)\right|
\leqslant\frac{1}{4}\|u^l\|_{\dot{H}^{2\alpha+1}}^2+
C(l)(\|S_lu^l\|^2+\|S_lu^l\|_{p+1}^{2p}),\nonumber\\
&|(f^l, (-\Delta)^{2\alpha}u_t^l+\epsilon (-\Delta)^{2\alpha}u^l)|\leqslant \frac{1}{4}(\|u_t^l\|_{\dot{H}^{3\alpha}}^2
+\epsilon\|u^l\|_{\dot{H}^{2\alpha+1}}^2)+C(l)\|f^l\|^2.
\end{align}
Inserting  \eqref{2.13*}-\eqref{2.15*} into \eqref{2.12*} and using \eqref{2.15}, \eqref{2.8*} turn out
\begin{eqnarray}
\frac{d}{dt}H_1(\xi_{u^l}(t))+\kappa  H_1(\xi_{u^l}(t))\leqslant C(l, R_0, \|f\|). \label{2.16*}
\end{eqnarray}
Let
\[\Psi(\xi_{u^l})=H(\xi_{u^l})+H_1(\xi_{u^l}).\]
The combination of  \eqref{2.12} and  \eqref{2.16*} gives
\begin{align*}
&\frac{d}{dt}\Psi(\xi_{u^l})+\kappa \Psi(\xi_{u^l})\leqslant C(l, R_0, \|f\|),\nonumber\\
&
\|\xi_{u^l}(t)\|_{X_\alpha}^2=\|u^l(t)\|_{H^{2\alpha+1}}^{2}+\|u_t^l(t)\|^2_{H^{2\alpha}}\leqslant C(l, R_0, \|f\|),\ \ t>0,
\end{align*}
which means that formula \eqref{e2.19} holds and hence $T_l=+\infty$.

For any $\psi\in  L^2\cap L^{p+1}$, by Lemma \ref{L2.4} and \eqref{2.15},
\begin{align}&
|(g(S_lu^l), \psi)|\leqslant C(\|S_lu^l\|\|\psi\|+\|S_lu^l\|^p_{p+1}\|\psi\|_{p+1})
\leqslant C(\|u^l\|+\|S_lu^l\|^p_{p+1})\|\psi\|_{L^2\cap L^{p+1}},\nonumber\\
&\|g(S_lu^l)\|_{L^2+ L^{1+\frac{1}{p}}}\leqslant C(\|u^l\|+\|S_lu^l\|^p_{p+1})\leqslant C(R_0,\|f\|),\label{2.27*}\\
&\|S_l(g(S_lu^l))\|_{H^{-2\alpha}}\leqslant \|g(S_lu^l)\|_{H^{-2\alpha}}\leqslant C\|g(S_lu^l)\|_{L^2+ L^{1+\frac{1}{p}}}\leqslant C(R_0,\|f\|),\nonumber
\end{align}
where we have used the facts: $L^2+ L^{1+\frac{1}{p}} =(L^2\cap L^{p+1})'$ and $ H^{2\alpha}\hookrightarrow (L^2\cap L^{p+1})$ for $1\leqslant p< p_\alpha$. So  by  Eq. \eqref{2.9''} and \eqref{2.15},
\begin{align}
\int^{t+1}_t\|u_{tt}^l(\tau)\|_{H^{-2\alpha}}^2d\tau \leqslant C(R_0,\|f\|),\ \ t\geqslant 0.  \label{2.17}
\end{align}

Owing to $(u^l, u_t^l, u_{tt}^l)\in C(\mathbb{R}^+; H^{2\alpha+1}\times H^{2\alpha}\times L^2)$,  differentiating Eq. \eqref{2.9''} with respect to $t$  and letting   $v^l=u_t^l$, we see that $v^l$ solves
\begin{equation}
v^l_{tt}+(I-\Delta) v^l+(I+(-\Delta)^\alpha) v^l_{t}+S_l(g'(S_lu^l)S_lv^l)=0. \label{2.18}
\end{equation}
Using the multiplier $(I-\Delta)^{-\alpha}v_{t}^l+\epsilon v^l$ in $\eqref{2.18}$ gives
\begin{align}&
\frac{1}{2}\frac{d}{dt}H_2(\xi_{v^l}(t))
+\|(I+(-\Delta)^\alpha)^{1/2}(I-\Delta)^{-\alpha/2}v_{t}^l\|^{2}
+\epsilon\Big(\|v^l\|^2_{H^1}
+\int S_l(g'(S_lu^l)S_lv^l)v^ldx\Big)\nonumber\\
=&\epsilon\|v_{t}^l\|^{2}-\int S_l(g'(S_lu^l)S_lv^l)(I-\Delta)^{-\alpha} v_t^ldx,\label{2.19}
\end{align}
where $\xi_{v^l}(t)=(v^l(t), v^l_t(t))$,
\begin{align}\label{2.20}&
H_2(\xi_{v^l})=\|v_{t}^l\|^2_{H^{-\alpha}}+\|v^l\|^2_{H^{1-\alpha}}
+\epsilon\left(\|v^l\|^2_{\dot{H}^\alpha}+\|v^l\|^2
+2(v^l,v_{t}^l)\right) \sim \|v_{t}^l\|^2_{H^{-\alpha}} +\|v^l\|^2_{H^\alpha},\\
 &\|(I+(-\Delta)^\alpha)^{1/2}(I-\Delta)^{-\alpha/2}v_{t}^l\|\sim \|v_{t}^l\|\nonumber
\end{align}
for  $\epsilon>0$ suitably small, and  where the equivalent constants are independent of $l$. By Lemma \ref{L2.4},
\begin{align}
\int S_l(g'(S_lu^l)S_lv^l)v^ldx
=&\int g'( S_lu^l)(S_lv^l)^2dx\nonumber\\
\geqslant& C_1d_0\int |S_lu^l|^{p-1}(S_lv^l)^2dx-C_0\|S_lv^l\|^2,\label{2.21.}
\end{align}
and
 \begin{align}
&\left|\int S_l(g'(S_lu^l)S_lv^l)(I-\Delta)^{-\alpha}v_t^ldx\right|
=\left|\int g'(S_lu^l)S_lv^lS_l((I-\Delta)^{-\alpha}v_t^l)dx\right|\nonumber\\
\leqslant & C\|S_lv^l\|\|S_l((I-\Delta)^{-\alpha}v_t^l)\|\nonumber\\
&
+\begin{cases} C\|S_lu^l\|^{p-1}_{p+1}
\|S_lv^l\|_{p+1}\|S_l((I-\Delta)^{-\alpha}v_t^l)\|_{p+1},\ \ 1\leqslant  p\leqslant  p^*,\\
\frac{C_1\epsilon}{2}\int|S_lu^l|^{p-1}|S_lv^l|^2dx
+C(\epsilon)\int|S_lu^l|^{p-1}|S_l((I-\Delta)^{-\alpha}v_t^l)|^2dx,\ \  p^*< p< p_\alpha
\end{cases}\nonumber\\
\leqslant &  C\|S_lv^l\|\| v_t^l\|_{H^{-2\alpha}}\nonumber\\
&
+\begin{cases}C(R_0, \|f\|)
\|S_lv^l\|_{H^1}\|S_l((I-\Delta)^{-\alpha}v_t^l)\|_{p+1},\ \ 1\leqslant  p\leqslant  p^*,\\
\frac{C_1\epsilon}{2}\int|S_lu^l|^{p-1}|S_lv^l|^2dx
+C(\epsilon)\|S_lu^l\|^{p-1}_{p+1}
\|S_l((I-\Delta)^{-\alpha}v_t^l)\|_{p+1}^2,\ \  p^*< p< p_\alpha,
\end{cases}\nonumber\\
\leqslant& \epsilon(\frac{1}{2} \|S_lv^l\|^2_{H^1}+\frac{d_0C_1}{2} \int|S_lu^l|^{p-1}|S_lv^l|^2dx+\|v_{t}^l\|^2)
+C(\epsilon, R_0, \|f\|)(\|S_lv^l\|^2+\|v_{t}^l\|_{H^{-2\alpha}}^2),\label{e2.49}
\end{align}
where we have used the Sobolev embedding$ H^{2\alpha-\delta}\hookrightarrow L^{p+1}$ for $\delta: 0<\delta\ll 1$  and the formula
\begin{equation*}
\|S_l((I-\Delta)^{-\alpha}v_t^l)\|^2_{p+1}\leqslant C\|S_l((I-\Delta)^{-\alpha}v_t^l)\|^2_{H^{2\alpha-\delta}}
\leqslant C\|v_{t}^l\|^2_{H^{-\delta}} \leqslant \epsilon \|v_{t}^l\|^2
+C(\epsilon)\|v_{t}^l\|^2_{H^{-2\alpha}}. \label{e2.50}
\end{equation*}
  Inserting  \eqref{2.21.}-\eqref{e2.49}    into \eqref{2.19} and using  \eqref{2.20} we obtain
\begin{align}&
 \frac{d}{dt}H_2(\xi_{v^l})+\kappa\Big(H_2(\xi_{v^l})
 +2\|v_{t}^l\|^{2}+\|v^l\|^2_{H^{1}}
 +d_0\int|S_lu^l|^{p-1}|S_lv^l|^2dx\Big)\nonumber\\
\leqslant& C(R_0, \|f\|)
(\|u_t^l\|^2+\|u_{tt}^l\|_{H^{-2\alpha}}^2).\label{2.22}
\end{align}
When $0\leqslant t\leqslant 1$, multiplying \eqref{2.22}  by $t^2$ yields
\begin{align}&
\frac{d}{dt}\Big(t^2H_2(\xi_{v^l})\Big)+\kappa t^2\left(H_2(\xi_{v^l})+2\|v_{t}^l\|^{2}\right)\nonumber\\
\leqslant& C(R_0, \|f\|) (\|u_t^l\|^2+\|u_{tt}^l\|_{H^{-2\alpha}}^2)+Ct(\|v^l\|_{H^\alpha}^{2}
+\|v_{t}^l\|^2_{H^{-\alpha}})\nonumber\\
\leqslant& \kappa t^2\|v_{t}^l\|^{2}+C(\kappa, R_0, \|f\|)(\|u_t^l\|_{H^\alpha}^{2}
+\|u_{tt}^l\|_{H^{-2\alpha}}^2),\label{2.23}
\end{align}
where we have used the fact
\[Ct\|v_{t}^l\|^2_{H^{-\alpha}}\leqslant Ct\|v_t^l\|\|v_t^l\|_{H^{-2\alpha}}\leqslant
\kappa t^2\|v_{t}^l\|^{2}+C(\kappa)\|v_t^l\|^2_{H^{-2\alpha}}.\]
Applying the Gronwall inequality  to \eqref{2.23} over $(0, t]$, with $ 0<t\leqslant1$, and utilizing \eqref{2.16}, \eqref{2.17}, we have
\begin{eqnarray}&&
t^2H_2(\xi_{v^l}(t))\leqslant  C(R_0, \|f\|)\int^1_0(\|u_t^l(\tau)\|_{H^\alpha}^{2}
+\|u^l_{tt}(\tau)\|_{H^{-2\alpha}}^2)d\tau
\leqslant C(R_0, \|f\|), \nonumber\\
&&\|u_t^l(t)\|_{H^\alpha}^2+\|u^l_{tt}(t)\|_{H^{-\alpha}}^2\leqslant \frac{1}{t^2}C(R_0,\|f\|),
 \ \ 0<t\leqslant 1.\label{2.24}
\end{eqnarray}
When $t\geqslant 1$, applying the Gronwall inequality to \eqref{2.22} over $(1, t)$ and using \eqref{2.24} at $t=1$, we obtain
 \begin{align}
\|u_t^l(t)\|_{H^\alpha}^2+\|u_{tt}^l(t)\|_{H^{-\alpha}}^2
\leqslant& H_2(\xi_{u_t^l}(1))e^{-\kappa(t-1)}+C\int^t_1e^{-\kappa(t-\tau)}
(\|u_t^l\|_{H^\alpha}^{2}+\|u_{tt}^l\|^2_{H^{-2\alpha}})d\tau\nonumber\\
\leqslant& C(R_0, \|f\|).  \label{2.25}
\end{align}
Integrating  \eqref{2.22} over $(t,t+1)$, with $t> 0$,  gives
\begin{align}
\int^{t+1}_t(d_0\int|S_lu^l(\tau)|^{p-1}|S_lu^l_t(\tau)|^2dx
+\|u_{t}^l(\tau)\|^{2}_{H^1}+\|u_{tt}^l(\tau)\|^{2})d\tau\leqslant(1+\frac{1}{t^2})C(R_0, \|f\|).\label{2.26}
\end{align}
The combination of \eqref{2.24}-\eqref{2.26} means that  estimate \eqref{2.3} uniformly holds for $U^l$.

Using the multiplier $-\Delta u^l$ in Eq. \eqref{2.9''} we get
\begin{align*}&
\frac{1}{2}\frac{d}{dt}H_3(u^l)+\|u^l\|^2_{\dot{H}^2}
+\|u^l\|^2_{\dot{H}^1} +(S_l(g(S_lu^l)),  -\Delta u^l)
 =(f^l-u^l_{tt}, -\Delta u^l)+(u^l, u_t^l),\label{2.27}
\end{align*}
where
\[H_3(u^l)=\| u^l\|^2_{\dot{H}^{1+\alpha}}+\| u^l\|^2_{\dot{H}^1}+\|u^l\|^2\sim \|u^l\|^2_{H^{1+\alpha}}.\]
Since $g'(s)> -C_0$,
\begin{align*}
 (S_l(g(S_lu^l)), -\Delta u^l)=(g'(S_lu^l)\nabla S_lu^l, \nabla S_lu^l)\geqslant& -C_0\|u^l\|^2_{H^1},
\end{align*}
we get
\begin{align}&
\frac{d}{dt}H_3(u^l)
+\|u^l\|^2_{\dot{H}^2}+\|u^l\|^2_{\dot{H}^1}+\|u^l\|^2\leqslant C
\|u^l_{tt}\|^2+C(R_0, \|f\|).\label{2.28}
\end{align}
Multiplying \eqref{2.28} by $(t-\frac{a}{2})^\frac{1}{1-\alpha}$ with $\frac{a}{2}\leqslant t\leqslant 1+\frac{a}{2}, 0< a\ll 1$, we get
\begin{align*}
&\frac{d}{dt}\left((t-\frac{a}{2})^\frac{1}{1-\alpha}H_3(u^l)\right)
+(t-\frac{a}{2})^\frac{1}{1-\alpha}\|u^l\|^2_{H^2}
\leqslant C\|u_{tt}^l\|^2
+\frac{C(t-\frac{a}{2})^{\frac{\alpha}{1-\alpha}}}{1-\alpha}
\|u^l\|^2_{H^{1+\alpha}}+C(R_0, \|f\|).
\end{align*}
By the interpolation theorem,
\begin{align*}
\frac{C}{1-\alpha}(t-\frac{a}{2})^{\frac{\alpha}{1-\alpha}}\|u^l\|^2_{H^{1+\alpha}}\leqslant&
\frac{C}{1-\alpha} (t-\frac{a}{2})^{\frac{\alpha}{1-\alpha}}\|u^l\|_{H^2}^{2\alpha}
\|u^l\|_{H^1}^{2(1-\alpha)}
\\\leqslant& \frac{(t-\frac{a}{2})^\frac{1}{1-\alpha}}{2}\|u^l\|^2_{H^2}
+C(\alpha)\|u^l\|^2_{H^{1}},
\end{align*}
where
$C(\alpha)= (\frac{C\alpha}{1-\alpha})^\frac{\alpha}{1-\alpha}$, so
\begin{align}
&\frac{d}{dt}\left((t-\frac{a}{2})^\frac{1}{1-\alpha}H_3(u^l)\right)
+\kappa(t-\frac{a}{2})^\frac{1}{1-\alpha}H_3(u^l)
\leqslant C\|u^l_{tt}\|^2
+C(\alpha)\|u^l\|^2_{H^{1}}+C(R_0, \|f\|),\nonumber\\
&\|u^l(t)\|_{H^{1+\alpha}}^2\leqslant C(\alpha)C(R_0, \|f\|)(t-\frac{a}{2})^{-\frac{1}{1-\alpha}},
   \ \ a/2<t\leqslant 1+a/2. \label{2.29}
\end{align}
Applying the Gronwall lemma to \eqref{2.28} over $(1+a/2,t)$ and making use of  \eqref{2.26} and  \eqref{2.29} at $t=1+a/2$, we get
\begin{align}
 \|u^l(t)\|_{H^{1+\alpha}}^{2}
\leqslant & C\|u^l(1+a/2)\|_{H^{1+\alpha}}^{2}+C\int^t_{1+a/2}e^{-\kappa(t-\tau)}
 \left(\|u^l_{tt}(\tau)\|^2 +C(R_0, \|f\|)\right)d\tau\nonumber\\
\leqslant & C(\alpha)C(R_0, \|f\|),\ \ t>1+a/2. \label{2.30}
\end{align}
Integrating  \eqref{2.28} over $(t,t+1)$, with $t> a/2$,  and exploiting \eqref{2.29}-\eqref{2.30}, we obtain
\begin{align}
 \int^{t+1}_t\|u^l(\tau)\|_{H^{2}}^{2}d\tau
\leqslant C(\alpha)C(R_0, \|f\|)\Big(1+(t-\frac{a}{2})^{-\frac{1}{1-\alpha}}\Big). \label{2.30.}
\end{align}
The combination  \eqref{2.29}-\eqref{2.30.} implies that estimate \eqref{2.4} uniformly holds  for $U^l(t)$.

 By estimates  \eqref{2.15}-\eqref{2.16} and \eqref{2.27*},  we have   (subsequence if necessary)
\begin{align}&
u^l\rightarrow u\ \ weakly^* \ \ in \ \ L^{\infty}(0,T; H^{1}), \nonumber\\
&g(S_lu^l)\rightarrow \vartheta \ \ weakly^* \ \ in \ \ L^{\infty}(0,T; L^2+ L^{\frac{p+1}{p}}),\label{2.32.}\\
&u^l_{t}\rightarrow u_{t}\ \ weakly^* \ \ in \ \ L^{\infty}(0,T;L^2)\cap L^{2}(0,T;H^{\alpha}).\nonumber
\end{align}
If $\vartheta=g(u)$,  then $u^l$ satisfies that for any  $\phi\in H^1\cap L^{p+1}$,
\begin{align*}
&(u^l_t, \phi)+\int^t_0\left[(\nabla u^l, \nabla \phi)+((-\Delta)^\frac{\alpha}{2} u^l_t,
(-\Delta)^\frac{\alpha}{2} \phi)
+(u^l_t+u^l, \phi)+(S_l(g(S_lu^l))-f^l, \phi)\right]d\tau
=(u^l_1, \phi).
\end{align*}
Letting $l\rightarrow \infty$, we get
\[(u_t, \phi)+\int^t_0\left[(\nabla u, \nabla \phi)+((-\Delta)^\frac{\alpha}{2} u_t, (-\Delta)^\frac{\alpha}{2} \phi)
+(u_t+u, \phi)+(g(u)-f, \phi)\right]d\tau=(u_1, \phi),\]
where we have used the   fact:
\begin{align*}
\lim_{ l\rightarrow \infty}\int^t_0(S_l(g(S_lu^l))-g(u), \phi)d\tau
=\lim_{ l\rightarrow \infty}\int^t_0\Big[(g(S_lu^l), S_l\phi-\phi)+(g(S_lu^l)-g(u),\phi)\Big]d\tau
= 0.
\end{align*}
Therefore, $u$ is the weak solution of problem \eqref{1.1}-\eqref{1.2}. By the lower semi-continuity of weak$^*$ limit, estimates \eqref{2.1}-\eqref{2.4} hold for $u$.

In order to show  $\vartheta=g(u)$, it is enough  to prove
\[S_lu^l\rightarrow u\ \  \text{a.e. in } [0, T]\times \mathbb{R}^N.\]
Let $B_R$ be a ball in $\mathbb{R}^N$ with center   zero and radius $R\in \mathbb{N}^+$.  Estimates \eqref{2.15} and \eqref{2.16} show  that the sequence $\{(S_lu^l, S_lu^l_t)\}$ is   uniformly bounded in the space
   \[L^\infty(0, T; H^1(B_R)\cap L^{p+1}(B_R))\times \left(L^\infty(0, T; L^2(B_R))\cap L^2(0, T; H^\alpha(B_R))\right).\]
By Lemma \ref{L2.1}, (subsequence if necessary)
\[S_lu^l \rightarrow u \ \ \text{in } C([0,T]; L^2(B_R)) \ \ \text{and a.e. in } [0, T]\times B_R.\]
By the standard diagonal argument, we can extract a subsequence (still denoted by itself) such that
\[ S_lu^l \rightarrow u \ \  \text{a.e. in } [0, T]\times \mathbb{R}^N,\]
which combining with  \eqref{2.32.}  implies that $\vartheta=g(u)$.

Let $u, v$ be  two solutions of problem \eqref{1.1}-\eqref{1.2} corresponding to initial data $(u_{0},u_{1})$
and $(v_{0},v_{1})$ with $\|(u_{0},u_{1})\|_{\mathcal{H}}+\|(v_{0},v_{1})\|_{\mathcal{H}}\leqslant R_0$,  respectively. Then $z=u-v$ solves
\begin{align}
&z_{tt}+(-\Delta)^\alpha z_{t}-\Delta z+z_{t}+z+g(u)-g(v)=0,\label{2.31}\\
&u(0)=u_{0}-v_0,\  u_{t}(0)=u_{1}-v_1.\nonumber
\end{align}
Using the multiplier  $(I-\Delta)^{-\alpha}z_t+\epsilon z$ in Eq. \eqref{2.31} yields
\begin{align}
&\frac{1}{2}\frac{d}{dt}H_2(\xi_{z})
+\|(I+(-\Delta)^\alpha)^{1/2}(I-\Delta)^{-\alpha/2}z_{t}\|^{2}
+\epsilon\left(\|z\|^{2}_{H^1}+(g(u)-g(v), z)\right)\nonumber\\
=&\epsilon\|z_{t}\|^{2}-(g(u)-g(v),(I-\Delta)^{-\alpha}z_t),\label{2.33}
\end{align}
where $\xi_{z}(t)=(z(t), z_t(t))$ and
\begin{align*}&
H_2(\xi_{z}(t))=\|z_{t}\|^{2}_{H^{-\alpha}}+\|z\|^2_{H^{1-\alpha}}
+\epsilon
\left(\|z\|^{2}_{\dot{H}^\alpha}+\|z\|^{2}+2(z,z_{t})\right)\sim \|z\|_{H^\alpha}^{2}+\|z_{t}\|^{2}_{H^{-\alpha}}
\end{align*}
for $\epsilon>0$ suitably small.  When $1\leqslant p< p_\alpha$,  similar to the proofs of \eqref{2.21.}-\eqref{e2.49}  we have
 \begin{align}&
(g(u)-g(v), z)\geqslant  d_0C_2\int(|u|^{p-1}+|v|^{p-1})|z|^2dx-C\|z\|^2,\label{2.35}\\
&\left|\left(g(u)-g(v), (I-\Delta)^{-\alpha}z_t\right)\right|\nonumber\\
\leqslant&  \epsilon\left(\frac{d_0C_2}{2}\int(|u|^{p-1}+|v|^{p-1})|z|^2dx
+\frac{1}{2}\|z\|^2_{H^1}+\|z_{t}\|^{2}\right)
+C(\epsilon)(\|z\|^2+\|z_{t}\|_{H^{-2\alpha}}^{2})\label{2.36}
\end{align}
for some $C_2> 0$.  Inserting \eqref{2.35}-\eqref{2.36} into \eqref{2.33}  arrives at
\begin{align}&
\frac{d}{dt}H_2(\xi_{z})+\kappa(\|z\|^2_{H^1}+\|z_{t}\|^{2})\leqslant C(\|z\|^2+\|z_{t}\|_{H^{-2\alpha}}^{2}).
\label{2.32}
\end{align}
Using the Sobolev embedding  $H^1\times L^2\hookrightarrow H^{\alpha}\times H^{-\alpha}$ and applying the Gronwall lemma to \eqref{2.32}, one obtains \eqref{2.6}. Then one   directly obtains \eqref{2.5}  for $H^{\alpha}\times H^{-\alpha}\hookrightarrow L^2\times H^{-2\alpha}$.
\end{proof}

\begin{rmk}\label{R2.3}
(i)
The control constants in Theorem  \ref{T2.2} are independent   of $\alpha$ except $C(\alpha)$. A simple calculation shows that $\displaystyle\lim_{\alpha\rightarrow 1} C(\alpha)=\infty$,
which means that estimate \eqref{2.4} does not hold for  $\alpha=1$.

(ii) Although by carefully choosing $a$ and $t$  in  \eqref{2.4},
for example
$t\geqslant \frac{a}{2}+(\frac{C\alpha}{1-\alpha})^\alpha$, we can  balance the blowup of the constant $C(\alpha)$ and obtain
\[\|u^\alpha(t)\|_{H^{1+\alpha}}\leqslant C(R_0, \|g\|), \ \  t\geqslant \frac{a}{2}+(\frac{C\alpha}{1-\alpha})^\alpha.\]
But the additional regularity for $u^\alpha$ still fails when $\alpha\rightarrow 1^-$ for $\frac{a}{2}+(\frac{C\alpha}{1-\alpha})^\alpha\rightarrow +\infty$.
 \end{rmk}

When $1\leqslant p< p_\alpha$, based on Theorem \ref{T2.2}, we   define the solution operator
$S^\alpha(t): \mathcal{H}\rightarrow \mathcal{H},$
 \[S^\alpha(t)(u_{0},u_{1})=\xi_{u^\alpha}(t)=(u^\alpha(t),u^\alpha_{t}(t)),\ (u_{0},u_{1})\in \mathcal{H},\]
where $u^\alpha$ is the solution of  problem \eqref{1.1}-\eqref{1.2}, and the family of solution operators $\{S^\alpha(t)\}_{t\geqslant 0}$ constitutes a  semigroup on $\mathcal{H}$ for each  $\alpha\in (1/2, 1)$,  which is Lipschitz continuous in weaker space $\mathcal{H}_{-\alpha}$. By the interpolation and standard argument as in \cite{Y-D-L}, one easily knows that $\{S^\alpha(t)\}_{t\geqslant 0}$ is H\"{o}lder continuous in phase space $\mathcal{H}$ for each  $\alpha\in (1/2, 1)$.

\section{Global attractor}\label{Sec3}

In this section, we study the existence of global attractor of the dynamical system $(S^\alpha(t), \mathcal{H})$ for each $\alpha\in (1/2, 1)$.  For   brevity, we denote $S^\alpha(t)$ by $S(t)$ and $(u^\alpha, u^\alpha_t)$ by $(u,u_t)$, respectively.

It follows  from  \eqref{2.1} that the dynamical system $(S(t), \mathcal{H})$ possesses a bounded absorbing set
  \[  B_0=\left\{(u,v)\in \mathcal{H}\mid\|u\|^{2}_{H^{1}}+\|u\|^2_{p+1}+\|v\|^{2}\leqslant R_1^2\right\}\]
for $R_1$ suitably large, so there exists a positive constant $t_0$ such that $S(t)B_0\subset B_0$  for $t\geqslant t_0$. Let
\begin{align*}
   \mathcal{B}=\Big[\bigcup_{t\geqslant t_0+1}S(t)B_0\Big]_{\cal H_{-\alpha}},
\end{align*}
where $[ \ ]_{\cal H_{-\alpha}}$ denotes the closure in  $\cal H_{-\alpha}$. Obviously, $\mathcal{B}$ is a forward invariant absorbing set and bounded in $\cal H_{\alpha}$ (see \eqref{e2.1}, \eqref{2.3} and \eqref{2.4}). Then the solution $u$ corresponding to the initial data $(u_0, u_1)\in \mathcal{B}$ satisfies
\begin{align*}&
\|(u,u_t)(t)\|_{\cal H_\alpha}^2+\int^\infty_0\|u_{t}(\tau)\|_{H^{\alpha}}^{2}d\tau
+\int^{t+1}_t(d_0\int|u(\tau)|^{p-1}|u_t(\tau)|^2dx
+\|u_t(\tau)\|^2_{H^1}+\|u_{tt}(\tau)\|^{2})d\tau\\
&\leqslant C(R_0, \|f\|), \ \ \forall t> 0,\\
&
\left|\int^t_0(g(u), u_t)d\tau\right|
\leqslant C\int^t_0\left(\|u(\tau)\|\|u_t(\tau)\|+\|u(\tau)\|^{p+1}_{p+1}
+\int|u(\tau)|^{p-1}|u_t(\tau)|^2dx\right)d\tau\leqslant C(t), \ \forall t> 0,
\end{align*}
i.e.,  $g(u)u_t\in L^1([0, t]\times \R^N)$. So by the technique  used in \cite{St}, we can  use the multiplier $u_t$ in Eq. \eqref{1.1} and   the energy equality
\begin{equation*}
E(\xi_u(t))+\int^t_s(\|u_t(\tau)\|^2_{\dot{H}^\alpha}
+\|u_t(\tau)\|^2)d\tau=E(\xi_u(s)), \ \ \forall t> s\geqslant 0\label{3.1}
\end{equation*}
holds, where $\xi_u=(u,u_t)$,
\[E(u, v)=\frac{1}{2}(\|v\|^2+\|u\|_{\dot{H}^1}^2+\|u\|^2)
+\int G(u)dx-(f, u).\]

We construct the function
\begin{align*}
 K_{0}(s)=
 \begin{cases}
0,&\ \ 0\leqslant s\leqslant 1,\\
s-1,&\ \ 1<s\leqslant 2,\\
1,&\ \ s>2,
\end{cases}
 \end{align*}
and let
\begin{align*}
 K_{\delta}(s)=(\rho_{\delta}*K_{0})(s)
 =\int_{\mathbb{R}}\rho_{\delta}(s-y)K_{0}(y)dy,
\end{align*}
where $\rho_{\delta}(s)$ is the standard mollifier on $\mathbb{R}$ with $supp\rho_{\delta}\subset[-\delta,\delta]$. Obviously,
\begin{align*}
K_{\delta}\in C^{\infty}(\mathbb{R}),\ \  0\leqslant K_{\delta}(s)\leqslant 1, \ \ K_{\delta}(s)=0\
\ as \ \ 0\leqslant s<1; \ K_{\delta}(s)=1\ \  as\ \ s>2,
\end{align*}
with $0<\delta\ll1$. Let $\psi(x)=K_{\delta}(\frac{|x|}{R})$. A simple calculation shows that
\begin{align}
\psi(x)=0\  \ \text{as} \ |x|<R, \ \ 0\leqslant \psi(x)\leqslant 1\ \  \text{and} \ \ |\nabla \psi(x)|\leqslant
\begin{cases}
 CR^{-1}, & R\leqslant |x|\leqslant 2R, \\
0, &\text{others},
\end{cases}  \label{3.2}
 \end{align}
where the constant $C$ is independent of $R$. Hence, we have
\begin{align*}
\|\nabla \psi\|_p\leqslant CR^{-1+\frac{N}{p}},\ \ p\geqslant 1.
\end{align*}
In light of $\psi\in L^\infty$, by the Sobolev embedding theorem we have
\begin{align} &
\|\psi\Lambda^\alpha\varphi\|_\frac{2N}{\alpha}\leqslant C\|\psi\|_{\dot{H}^{\alpha, \frac{2N}{\alpha}}}\leqslant C \|\psi\|_{\dot{H}^{1, \frac{2N}{2-\alpha}}}\leqslant CR^{-\frac{\alpha}{2}},\label{3.4a}
\end{align}
where $\Lambda=(-\Delta)^\frac{1}2$. Using the similar procedure as to the estimate \eqref{3.4a}, we have
\begin{eqnarray}\label{3.4}
\|\psi\|_{\dot{H}^{s,q}}\leqslant C\|\psi\|_{\dot{H}^{1,\frac{Nq}{(1-s)q+N}}}
\leqslant CR^{-(s-\frac{N}{q})}, \ \  \forall s\leqslant 1, q\geqslant 1.
\end{eqnarray}

\begin{lem}\label{L3.5.}\cite{Kenig}
Assume that  $\Lambda=(-\Delta)^\frac{1}2,  0< s< 1, s_1, s_2\in [0,s],s=s_1+s_2$ and $p, p_1, p_2\in (1,+\infty)$ satisfying
\[\frac{1}{p}=\frac{1}{p_1}+\frac{1}{p_2},\]
then
\begin{align*}
\|\Lambda^s(fg)-f\Lambda^sg-g\Lambda^sf\|_{p}\leqslant C\|\Lambda^{s_1}f \|_{ p_1}\|\Lambda^{s_2}g\|_{ p_2 }.
\end{align*}
Moreover, the inequality holds for $ s_1 = 0, p_1 = \infty, i.e., $
\[\|\Lambda^s(fg)-f\Lambda^sg-g\Lambda^sf\|_{p}\leqslant C\| f \|_{\infty}\|\Lambda^{s}g\|_{ p }.\]
\end{lem}

\begin{lem}\label{L3.3}
Let $u$ be the solution of   problem \eqref{1.1}-\eqref{1.2} as shown in Theorem \ref{T2.2}. Then
\begin{align}\label{3.5}
\|\psi \Lambda^\alpha u\|\leqslant C(\|\psi u\|+\|\psi \nabla u\|)+CR^{-\frac{\alpha}{2}}\|u\|_{H^1}.
\end{align}
\end{lem}

\begin{proof}
Since
\begin{align}\label{3.6}
\|\psi \Lambda^\alpha u\|\leqslant \|\Lambda^\alpha(\psi u)-\psi\Lambda^\alpha u\|
+\|\Lambda^\alpha(\psi u)\|,
\end{align}
by Lemma \ref{L3.5.} (taking  $s_1=\alpha, s_2=0$ there),  the H\"{o}lder inequality and  \eqref{3.2}-\eqref{3.4a},  we have
\begin{align*}
&\|\Lambda^\alpha (\psi u)-\psi \Lambda^\alpha u\|\leqslant \|u\Lambda^\alpha \psi \| +C
\| \Lambda^\alpha\psi\|_{\frac{2N}{\alpha}}
\|u\|_{\frac{2N}{N-\alpha}}
\leqslant C
\| \psi\|_{\dot{H}^{\alpha,\frac{2N}{\alpha}}}
\|u\|_{\frac{2N}{N-\alpha}}
\leqslant C R^{-\frac{\alpha}{2}} \|u\|_{H^1},\\
&
\|\Lambda^\alpha (\psi u)\|=\|\psi u\|_{\dot{H}^\alpha}\leqslant C\|\psi u\|^{1-\alpha}
\|\psi u\|^\alpha_{\dot{H}^1}\leqslant C(\|\psi u\|+\|\psi\nabla u\|)+CR^{-1}\|u\|.
\end{align*}
Inserting above inequalities  into \eqref{3.6} yields estimate \eqref{3.5}.
\end{proof}

\begin{lem} \label{3.3**} Let $S(t)(u_{0},u_{1})=(u(t),u_{t}(t))$, with $(u_{0},u_{1})\in \mathcal{B}$. Then for any $\epsilon>0$, there exist positive constants $K=K (\epsilon)$ and $T_{0}=T_{0}(R_{1})$ such that
\begin{align*}
\|(u(t),u_{t}(t))\|_{\cal H(B^{C}_{2R})}<\epsilon\ \ as \ \ R\geqslant K, \ t\geqslant T_{0},
\end{align*}
where $B_{R}$ is the ball centered at zero with radius $R$ in $\mathbb{R}^N,  B_{R}^{C}=\mathbb{R}^N\setminus B_{R}$.
 \end{lem}

\begin{proof}
Using the multiplier $\psi^{2}(u_{t}+\epsilon u)$ in Eq. \eqref{1.1} and making use of the boundedness of $(u,u_t)$ in $\mathcal{H}_{\alpha}$,  we have
\begin{align}
\frac{d}{dt}H_4(\xi_u)+\Phi_2(\xi_u)
=&-2(\psi\nabla\psi (u_{t}+\epsilon u), \nabla u)
+\left(\Lambda^\alpha u_t, \psi^2\Lambda^\alpha(u_{t}+\epsilon u)
-\Lambda^\alpha(\psi^2(u_t+\epsilon u))\right)\nonumber\\
\leqslant&\epsilon^2\|\psi\nabla u\|^{2}+C(\|u_t\nabla\psi \|^{2}
+\|u \nabla\psi \|^{2})\nonumber\\
&  +\Big(\|\psi^2\Lambda^\alpha u_{t}-\Lambda^\alpha(\psi^2u_t)\|+\epsilon\|\psi^2\Lambda^\alpha u-\Lambda^\alpha(\psi^2u)\|\Big)\|\Lambda^\alpha u_t\|\nonumber\\
\leqslant&\epsilon^2\|\psi\nabla u\|^{2}+C(\|u_t\nabla\psi \|^{2}
+\|u \nabla\psi \|^{2})+CR^{-\frac{\alpha}{2}}(\|u_t\|_{H^\alpha}^2+\|u\|_{H^1}^2),\label{3.10}
\end{align}
where
\begin{align*}
H_4(\xi_u)=&\frac{1}{2}\Big(\|\psi u_{t}\|^2+\|\psi\nabla u\|^2+\|\psi u\|^2
+2\int\psi^2(G(u)-fu)dx\Big)\nonumber\\
 &+\epsilon\Big((\psi^2u,u_{t})+\frac{1}{2}(\|\psi u\|^{2}+\|\psi \Lambda^\alpha u\|^2)\Big),\\
\Phi_2(\xi_u) =&(1-\epsilon)\|\psi u_{t}\|^{2}+\|\psi \Lambda^\alpha u_{t}\|^{2}
+\epsilon (\|\psi\nabla u\|^2+\|\psi u\|^2+\int\psi^2g(u)udx-(\psi^2u,f)),
\end{align*}
and where we have used Lemma \ref{L3.5.} (with $s_1=\alpha, s_2=0$ there), and formula \eqref{3.4}  to get the estimates:
 \begin{align}&
\|\psi^2\Lambda^\alpha u_{t}-\Lambda^\alpha(\psi^2u_t)\|\leqslant \|u_{t}\Lambda^\alpha\psi^2\|+ C\|\Lambda^\alpha\psi^2\|_{\frac{2N}{\alpha}}
\|u_t\|_{\frac{2N}{N-\alpha}}\nonumber\\
\leqslant& C(\|\Lambda^\alpha\psi^2-2\psi\Lambda^\alpha\psi\|_{\frac{2N}{\alpha}}
+2\|\psi\Lambda^\alpha\psi\|_{\frac{2N}{\alpha}})
\|u_t\|_{\frac{2N}{N-\alpha}}\nonumber\\
\leqslant &C\|\psi\|_{ \frac{4N}{\alpha}}\|\psi\|_{\dot{H}^{\alpha, \frac{4N}{\alpha}}}\|u_t\|_{H^{\alpha/2}}
\leqslant CR^{-\alpha/2}\|u_t\|_{H^{\alpha}},\label{m3.7}
\end{align}
and by the similar argument as to  the estimate \eqref{m3.7}, we obtain
\[
\|\psi^2\Lambda^\alpha u-\Lambda^\alpha(\psi^2u)\|
\leqslant CR^{-\alpha/2}\|u\|_{H^1}.
\]
By  Lemma \ref{L3.3} and Remark \ref{R2.4},   we have
\begin{align}&
H_4(\xi_u) \geqslant \kappa(\|\psi u_{t}\|^2+\|\psi u\|^{2}+\|\psi\nabla u\|^2+d_0\int \psi^2|u|^{p+1}dx)
-C\|\psi f\|^2,\label{3.11.}\\
&\Phi_2(\xi_u)-\epsilon(\epsilon\|\psi\nabla u\|^{2}+ H_4(\xi_u))\nonumber\\
 \geqslant& \kappa(\|\psi  u_t\|^{2}+\|\psi \Lambda^\alpha u_t\|^{2}+\|\psi\nabla u\|^{2}+\|\psi u\|^{2})
 -CR^{-\frac{\alpha}{2}}\|u\|^2_{H^1}\label{3.12}
\end{align}
for $ \epsilon>0$ suitably small. Inserting \eqref{3.12} into \eqref{3.10} and using \eqref{3.2}, \eqref{3.11.}, we have
\begin{align}&
\frac{d}{dt}H_4(\xi_u)+\kappa H_4(\xi_u)\leqslant CR^{-\frac{\alpha}{2}}(\|u\|^2_{H^1}+\|u_t\|^2_{H^{\alpha}})+C\|\psi f\|^2, \nonumber\\
&H_4(\xi_u)\leqslant  C H_4(\xi_u(0))e^{-\kappa t}+C(R^{-\frac{\alpha}{2}}+\|f\|_{L^2(B^C_R)}^{2}). \label{3.13}
\end{align}
The combination of  \eqref{3.11.} and \eqref{3.13} implies the conclusion  of Lemma \ref{3.3**}.
 \end{proof}

\begin{theorem}\label{T3.7}
Let Assumption \ref{A2.3}  be valid.  Then the solution semigroup $S(t) (=S^\alpha(t))$ possesses a global attractor $\mathcal{A} (=\mathcal{A}_\alpha)$  in $\mathcal{H}$ for each $\alpha\in (1/2, 1)$.
\end{theorem}

 \begin{proof}
It is enough to show that  $S(t)$ is asymptotically compact on $\mathcal{B}$ with respect to  $\mathcal{H}$-topology.  Let $\{(u^{n}_{0},u^{n}_{1})\}$ be a bounded sequence in $\mathcal{B},
 S(t)(u^{n}_{0},u^{n}_{1})=(u^{n}(t),u^{n}_{t}(t))$.  Applying  \eqref{2.6} to
 \begin{align*}
 w^{m,n}(t)=u^{m}(t+t_{m}-T)-u^{n}(t+t_{n}-T), \ \ \hbox{with}\ \  t_{m}> t_{n}> T>0, t\geqslant0,
 \end{align*}
and making use of the fact $L^2\hookrightarrow H^{-2}$, we obtain
\begin{align*}
\|(w^{m,n},w^{m,n}_{t})(t)\|^{2}_{\mathcal{H}_{-\alpha}}
\leqslant Ce^{-\kappa t}+C\sup_{0\leqslant s\leqslant t}\|(w^{m,n},w^{m,n}_{t})(s)\|_{ L^2\times L^2}^{2}.
\end{align*}
Taking $t=T$  yields
\begin{align*}
&\|(u^{m}(t_{m})-u^{n}(t_{n}),u^{m}_{t}(t_{m})-u^{n}_{t}(t_{n}))\|^2_{\mathcal{H}_{-\alpha}}
\nonumber\\
\leqslant&  Ce^{-\kappa T}+C\big(\sup_{0\leqslant s\leqslant T}\|(w^{m,n},w^{m,n}_{t})(s) \|^{2}_{L^{2}(B_{2R}^{C})\times L^{2}(B_{2R}^{C})}\nonumber\\
&+\sup_{0\leqslant s\leqslant T}\|(w^{m,n},w^{m,n}_{t})(s)\|_{ L^{2}(B_{2R})\times L^{2}(B_{2R})}^2\big).\label{3.11}
\end{align*}
By  Lemma \ref{L2.1},
\begin{align*}&
\Pi_1=\left\{u\in L^\infty(0,T; H^{1+\alpha}(B_{2R}))\mid u_t\in L^2(0,T; H^\alpha(B_{2R}))\right\}\hookrightarrow\hookrightarrow C([0,T]; L^2(B_{2R})),\\
&\Pi_2=\left\{u_t\in L^\infty(0,T; H^\alpha(B_{2R}))\mid u_{tt}\in L^2(0,T;L^2(B_{2R}))\right\}\hookrightarrow\hookrightarrow C([0,T]; L^2(B_{2R})),
\end{align*}
thus the subsequence
\[\{(u^{m}(t_{m}), u^{m}_{t}(t_{m}))\}  \ \text{is  precompact  in}\ \ C([0,T]; L^2(B_{2R})\times L^2(B_{2R})).\]
Therefore,  for any $\epsilon>0$,   fixing $T: Ce^{-\kappa T}<\epsilon/4$, there must exist a $N_0>0$ such that when  $m, n\geqslant N_0$, $t_{m}-T>T_{0}, t_{n}-T>T_{0}$, and
\[\|(u^{m}(t_{m})-u^{n}(t_{n}),u^{m}_{t}(t_{m})-u^{n}_{t}(t_{n}))\|_{\mathcal{H}_{-\alpha}}
  <\frac{\epsilon}{4}+\frac{\epsilon}{4}+\frac{\epsilon}{4}<\epsilon,\]
i.e.,  the semigroup $S(t)$ is asymptotically compact on $\mathcal{B}$ with respect to the topology $\cal H_{-\alpha}$.
Taking account of  the boundedness of $\cal B$ in $\cal H_\alpha(\hookrightarrow \mathcal{H}\hookrightarrow H_{-\alpha})$, by the interpolation one easily sees that   $S(t)$ is
asymptotically compact on $\mathcal{B}$ with respect to  $\mathcal{H}$-topology. Therefore, $(S(t), \cal H)$ possesses a global attractor $\cal{A}$, and $\cal{A}\subset \cal B$ is bounded in $\cal H_{\alpha}$.
\end{proof}

\section{Upper semicontinuity of the   global attractors}\label{Sec4}

\begin{lem}\cite{WWQ} \label{L4.1}
Let $X, Y$ be two Banach spaces,  $X\hookrightarrow Y$, and  the semigroup $\{S^\alpha(t)\}_{t\geqslant 0}$ has a bounded  absorbing set $B_\alpha$ and a global attractor $\cal A_\alpha$ in $X$ for each $\alpha\in I$ (a subset of  $\mathbb{R}$). Assume that  the following assumptions hold:

(i) the union $\displaystyle\cup_{\alpha\in I}B_\alpha$ is bounded in $X$;

(ii) for any sequences $\{\alpha_n\}\subset I, \{x_n\}\subset \displaystyle\cup_{\alpha\in I}B_\alpha$
and   $t_n\rightarrow \infty $, the sequence $\{S^{\alpha_n}(t_n)x_n\}$ is precompact in $X$;

(iii) for any sequences $\{\alpha_n\}\subset I$ with $\alpha_n\rightarrow \alpha_0$ and $\{x_n\}\subset X$ with $x_n\rightarrow x_0$ in $X$,
\[\lim_{n\rightarrow\infty}\|S^{\alpha_n}(t)x_n- S^{\alpha_0}(t)x_0\|_Y=0, \ \ \forall t\geqslant  0.\]
Then
\[\lim_{\alpha\rightarrow \alpha_0}dist_X\{\cal A_\alpha,\cal A_{\alpha_0}\}=0.\]
\end{lem}

 \begin{theorem}\label{T4.2}
Let Assumption \ref{A2.3} be valid, with  $1\leqslant p< p_{\alpha_0}\equiv \frac{N+4\alpha_0}{(N-4\alpha_0)^+}, \alpha_0\in I\equiv(1/2, 1)$. Then the family of global attractors  $\{\mathcal{A}_\alpha\}_{\alpha\in I}$  as shown in Theorem \ref{T3.7} is upper semicontinuous at the point $\alpha_0$, i.e.,
\begin{align}\label{Equ4.1}
\lim_{\alpha\rightarrow \alpha_0}dist_{\cal H}\{\mathcal{A}_\alpha, \mathcal{A}_{\alpha_0}\}=0.
\end{align}
\end{theorem}

\begin{proof}
Without loss of generality we assume that $\alpha\in [\gamma, \gamma_1]\equiv \Gamma (\subset (1/2, 1))$ for $\alpha\rightarrow \alpha_0$, where $\gamma\equiv \alpha_0-\eta (>0), \gamma_1\equiv \alpha_0+\eta$,  with  $\eta: 0<\eta< \min\{\alpha_0-1/2, \alpha_0/3, (1-\alpha_0)/3\}$.

Similar to the proof of  \eqref{2.17}, for the approximate solution
$u^l (=u^{\alpha, l})$ of the auxiliary problem  \eqref{2.9''}, the following estimate holds:
\[\int^{t+1}_t\|u^{l}_{tt}(\tau)\|^2_{H^{-2\gamma}}d\tau
\leqslant C(R_0,\|f\|),\ \ \forall t>0, \alpha\in \Gamma.\]
When $1\leqslant p< p_\gamma(\equiv \frac{N+4\gamma}{(N-4\gamma)^+})$, similar to the proofs of \eqref{2.3} and \eqref{2.4} (replacing the multiplier $(I-\Delta)^{-\alpha}v^{l}_{t}+\epsilon v^{ l}$ after Eq. \eqref{2.18} by $(I-\Delta)^{-\gamma}v^{l}_{t}+\epsilon v^{ l}$) one easily obtains that
\begin{align}&
\|u^{l}_{t}(t)\|_{H^{\alpha}}^{2}+\|u^{l}_{tt}(t)\|_{H^{-\gamma}}^{2}
+\int^{t+1}_t(\|u^{l}_{t}(\tau)\|_{H^1}^2+\|u^{l}_{tt}(\tau)\|^2)d\tau
\leqslant \left(1+\frac{1}{t^2}\right)C(R_0, \|f\|)\label{4.2}
\end{align}
and
\begin{align}
&\|u^{l}(t)\|_{H^{1+\gamma}}^{2} \leqslant \|u^{l}(t)\|_{H^{1+\alpha}}^{2}
\leqslant \Big(1+\frac{1}{(t-\frac{a}{2})^\frac{1}{1-\alpha}}\Big)
C(\alpha)C(R_0, \|f\|), \ \ t> a. \label{4.3}
\end{align}
A simple calculation shows that  $C(\alpha)= (\frac{C\alpha}{1-\alpha})^\frac{\alpha}{1-\alpha}\leqslant C_3(\alpha_0)$ for  $\alpha\in \Gamma$. So   by the lower semicontinuity of   weak limit, estimates \eqref{4.2} and \eqref{4.3} (replacing the control constant $C(\alpha)$ there by $C_3(\alpha_0)$) uniformly ($w.r.t.\ \alpha\in \Gamma$)  hold for the weak solutions $u^\alpha$ of    problem \eqref{1.1}-\eqref{1.2}.

(i) It follows from  estimate \eqref{2.1}  that the family of dynamical systems  $(S^\alpha(t), \cal H), \alpha\in \Gamma$  possesses in $\cal H$ a common absorbing set
\begin{align*}
B_1\equiv\{(u, v)\in \cal H, \|(u, v)\|_{\cal H}\leqslant R_1\}\ \ \hbox{for}\ \ R_1\geqslant C(\|f\|),
\end{align*}
and  there exists a  $t_1=t(R_1, \|f\|)$ such that $\cup_{\alpha\in \Gamma} S^\alpha(t)B_1\subset B_1$ for  $t\geqslant t_1$. Let
\begin{align*}
\cal B_1=\bigcup_{\alpha\in \Gamma}\bigcup_{t\geqslant t_1+2}S^\alpha(t)B_1 (\subset B_1).
\end{align*}
Then the set $\cal B_1$ is bounded in $\cal H_{\alpha}$  (see estimates \eqref{4.2}-\eqref{4.3}). Moreover, $\cal B_1$ is a common forward invariant absorbing set of  the family of dynamical  systems $(S^\alpha(t), \cal H), \alpha\in \Gamma$.   And $\mathcal{B}_1=\cup_{\alpha\in \Gamma} \mathcal{B}_1$.

(ii) For any sequences $\{\alpha_n\}\subset \Gamma$,  $\{\xi_{u^n}\}\subset \cal B_1(\equiv \bigcup_{\alpha\in \Gamma}\cal B_1) $ and $t_n\rightarrow \infty $,  the sequence $\{S^{\alpha_n}(t_n)\xi_{u^n}\}$ is precompact in $\cal H$.

Indeed, let  $S^\alpha(t)(u_{0},u_{1})=(u^\alpha(t),u^\alpha_{t}(t))$, with $(u_{0},u_{1})\in \mathcal{B}_1$ and $\alpha\in \Gamma$.
It follows from Lemma \ref{3.3**} that for any $\epsilon>0$, there exist positive constants $K_{2}=K_{2}(\epsilon)> 0$ and $T_{2}=T_2(R_1)> 0$ such that for all $\alpha\in \Gamma$,
\begin{align*}
\|(u^\alpha(t),u^\alpha_{t}(t))\|_{\mathcal{H}(B^{C}_{K})}<\epsilon\ \ as\ \ K\geqslant K_{2}, \ t\geqslant T_{2}.
\end{align*}
So there exists a $N>0$ such that when $n, m\geqslant N$, $t_n, t_m\geqslant T_2$ and
\[\|S^{\alpha_n}(t_n)\xi_{u^n}\|_{\mathcal{H}(B^{C}_{K_2})}<\epsilon.\]
Therefore,
\begin{align*}
&\|S^{\alpha_n}(t_n)\xi_{u^n}-S^{\alpha_m}(t_m)\xi_{u^m}\|_{\mathcal{H}}
\leqslant \|S^{\alpha_n}(t_n)\xi_{u^n}-S^{\alpha_m}(t_m)\xi_{u^m}\|_{\mathcal{H}(B_{K_2})}+2\epsilon, \ \ \forall n, m \geqslant N,
\end{align*}
which, combining with the boundedness of $\cal B_1$ in $\cal H_{\gamma}$ and $\mathcal{H}_{\gamma}(B_{K_2})\hookrightarrow\hookrightarrow \mathcal{H}(B_{K_2})$ for $1\leqslant p< p_\gamma$,   implies that the sequence
\[\{S^{\alpha_n}(t_n)\xi_{u^n}\} \ \ \text{is precompact  in } \cal H.\]

(iii) When $1\leqslant p< p_\gamma$,  for any sequences $\{\alpha_n\}\subset \Gamma$  with $\alpha_n\rightarrow \alpha_0$ and $\{\xi_{u^n}\}\subset \cal H (\hookrightarrow \cal H_{-1/2})$ with $\xi_{u^n}\rightarrow \xi$ in $\cal H$,  we have
\begin{align}
 \lim_{n\rightarrow\infty}
\|S^{\alpha_n}(t)\xi_{u^n}-S^{\alpha_0}(t)\xi\|_{\cal H_{-1/2}}=0, \ \ \forall t\geqslant 0.\label{e4.4}
\end{align}

Indeed, obviously formula \eqref{e4.4} holds  for $t=0$.  We  show that  formula \eqref{e4.4} holds for any  $t>0$. Let
\[S^{\alpha_n}(t)\xi_{u^n}=\xi_{u^{\alpha_n}}(t)
=(u^{\alpha_n}(t),u^{\alpha_n}_t(t)), \ \  S^{\alpha_0}(t)\xi
=\xi_{u^{\alpha_0}}(t)=(u^{\alpha_0}(t),u^{\alpha_0}_t(t)),\]
then  $z(t)= u^{\alpha_n}(t)-u^{\alpha_0}(t)$ solves
\begin{align}\label{5.3}\begin{cases}
 z_{tt}+(-\Delta)^{\alpha_n} z_{t}+ z_t-\Delta z+z
+[(-\Delta)^{\alpha_n}-(-\Delta)^{\alpha_0}]u^{\alpha_0}_t
+g(u^{\alpha_n})-g(u^{\alpha_0})=0,\\
  \xi_z(0)=(z,z_{t})(0)=\xi_{u^n}-\xi.
\end{cases}
\end{align}
Using the multiplier  $(I-\Delta)^{-\gamma}z_t+\epsilon z$ in Eq. \eqref{5.3} turns out
\begin{align}
&\frac{d}{dt}H_5(\xi_z)
+\left\|\left(I+(-\Delta)^{\alpha_n}\right)^\frac{1}{2}
(I-\Delta)^{-\frac{\gamma}{2}}z_t\right\|^2-\epsilon\|z_t\|^2
+\epsilon\|z\|^2_{H^1}
+\epsilon(g(u^{\alpha_n})-g(u^{\alpha_0}), z)\nonumber\\
=&-\left(g(u^{\alpha_n})-g(u^{\alpha_0}), (I-\Delta)^{-\gamma}z_t\right)
-\left([(-\Delta)^{\alpha_n}-(-\Delta)^{\alpha_0}]u^{\alpha_0}_t, (I-\Delta)^{-\gamma} z_t+\epsilon z\right), \label{5.4}
\end{align}
where $\xi_z=(z,z_t)$, and
\begin{align*}
H_5(\xi_z(t))=&\frac{1}{2}\left(\|z_t\|^2_{H^{-\gamma}}
+\|z\|^2_{H^{1-\gamma}}
+\epsilon(\|z\|^2_{\dot{H}^{\alpha_n}}+\|z\|^2)\right)
+\epsilon(z, z_t)
\sim  \|z\|^2_{H^{\alpha_n}}+\|z_t\|^{2}_{H^{-\gamma}}%\label{4.8.}
\end{align*}
for $\epsilon>0$ suitably small. Obviously,
\begin{align}\label{a4.8}&
\beta_1\|\xi_z\|_{\mathcal{H}_{-\gamma}}\leqslant \|z\|_{H^{\alpha_n}}+\|z_t\|_{H^{-\gamma}}\leqslant \beta_2\|\xi_z\|_\mathcal{H},\\
&\left\|\left(I+(-\Delta)^{\alpha_n}\right)^\frac{1}{2}
(I-\Delta)^{-\frac{\gamma}{2}}z_t\right\|\sim \|z_t\|_{H^{\alpha_n-\gamma}},\nonumber
\end{align}
with $\beta_1, \beta_2> 0$. Taking account of the Sobolev embedding $H^{\alpha_n-\gamma}\hookrightarrow H^{1-2\gamma} $ for $\eta< \alpha_0-1/2$, we have
\begin{align}\label{4.8*}&
\left|\left([(-\Delta)^{\alpha_n}-(-\Delta)^{\alpha_0}]u^{\alpha_0}_t, (I-\Delta)^{-\gamma}z_t+\epsilon z\right)\right|\nonumber\\
\leqslant&
\|[(-\Delta)^{\alpha_n}-(-\Delta)^{\alpha_0}]u^{\alpha_0}_t\|_{H^{-1}}
(\|(I-\Delta)^{-\gamma}z_t\|_{H^{1}}
+\|z\|_{H^{1}})\nonumber\\
\leqslant& \frac{\epsilon}{4}(\|z\|^2_{H^{1}}+\|z_t\|_{H^{\alpha_n-\gamma}}^2)
+C\|[(-\Delta)^{\alpha_n}-(-\Delta)^{\alpha_0}]u^{\alpha_0}_t\|_{H^{-1}}^2.
\end{align}
So when $1\leqslant p< p_{\gamma}(< p_{\alpha_0})$, inserting \eqref{4.8*} and \eqref{2.35}-\eqref{2.36} (replacing $u, v$ there by $u^{\alpha_n}, u^{\alpha_0}$, respectively) into \eqref{5.4}, we obtain
\begin{align}
\frac{d}{dt}H_5(\xi_z(t))+\kappa H_5(\xi_z(t))\leqslant& C\| [(-\Delta)^{\alpha_n}-(-\Delta)^{\alpha_0}]u^{\alpha_0}_t\|_{H^{-1}}^2
 +C(\|z\|^2+\|z_t\|^2_{H^{-2\gamma}})\nonumber\\
\leqslant& CH_5(\xi_z(t))+ C\|[(-\Delta)^{\alpha_n}-(-\Delta)^{\alpha_0}]u^{\alpha_0}_t\|_{H^{-1}}^2
,\label{4.7}
\end{align}
where we have used the Sobelev embedding $H^{\alpha_n}\times H^{-\gamma} \hookrightarrow L^2\times H^{-2\gamma}$.  Applying the Gronwall lemma to \eqref{4.7} and making use of \eqref{a4.8}, we receive
\begin{align}
\|\xi_z(t)\|^2_{\cal H_{-\gamma}}\leqslant Ce^{Ct}\|\xi_z(0)\|^2_{\cal H}+ C\int^t_0e^{C(t-\tau)}
\|[(-\Delta)^{\alpha_n}-(-\Delta)^{\alpha_0}]u^{\alpha_0}_t(\tau)\|_{H^{-1}}^2d\tau.
\label{4.8}
\end{align}
By \eqref{2.1} and   \eqref{4.2},  $u^{\alpha_0}_t\in L^2(0, t; H^{\alpha_0})\cap C_w([0, t]; H)$, and
\begin{align}&
 \|[(-\Delta)^{\alpha_n}-(-\Delta)^{\alpha_0}]u^{\alpha_0}_t(\tau)\|_{H^{-1}}=\Big(\int(1+|\xi|^2)^{-1}(|\xi|^{2\alpha_n}-|\xi|^{2\alpha_0})^2
|\widehat{u^{\alpha_0}_t}(\xi,\tau)|^2
 d \xi\Big)^{\frac{1}{2}}\nonumber\\
&\leqslant
 C\|u^{\alpha_0}_t(\tau)\|_{H^{2(\alpha_0+\eta)-1}}\leqslant C\|u^{\alpha_0}_t(\tau)\|_{H^{\alpha_0}}< \infty,\ \ a.e. \ \  \tau\in [0, t]
\label{4.9}
\end{align}
for   $\eta< (1-\alpha_0)/2$. Since
\begin{align*} &
\lim_{\alpha_n\rightarrow \alpha_0}(1+|\xi|^2)^{-1}(|\xi|^{2\alpha_n}-|\xi|^{2\alpha_0})^2
|\widehat{u^{\alpha_0}_t}(\xi,\tau)|^2=0 \ \ a.e. \ \ \text{in} \ \  \mathbb{R}^N\times [0,t],\\
&
(1+|\xi|^2)^{-1}(|\xi|^{2\alpha_n}-|\xi|^{2\alpha_0})^2
|\widehat{u^{\alpha_0}_t}(\xi,\tau)|^2\leqslant C(1+|\xi|^2)^{\alpha_0}|\widehat{u^{\alpha_0}_t}(\xi, \tau)|^2\in L^1(\mathbb{R}^N\times [0,t])
\end{align*}
(see \eqref{4.9}), by the Lebesgue dominated convergence theorem,
\begin{align}&
\lim_{\alpha_n\rightarrow \alpha_0}\int ^t_0e^{C(t-\tau)}\|[(-\Delta)^{\alpha_n}-(-\Delta)^{\alpha_0}]
u^{\alpha_0}_t(\tau)\|_{H^{-1}}^2d\tau
=0.\label{4.12}
\end{align}
Thus, by  \eqref{4.8} and \eqref{4.12},
\begin{align*}
\lim_{\alpha_n\rightarrow \alpha_0}\|(z, z_t)(t)\|_{\cal H_{-\gamma}}= 0.
\end{align*}
Taking account of the Sobolev embedding   $H^\gamma\hookrightarrow  H^{1/2}, L^2\hookrightarrow H^{-1/2}\hookrightarrow H^{-\gamma}$ and making use of the interpolation, we have
\begin{align}\label{4.13}
    \lim_{\alpha_n\rightarrow \alpha_0}\|(z, z_t)(t)\|_{\cal H_{-1/2}}\leqslant C\lim_{\alpha_n\rightarrow \alpha_0}\|(z, z_t)(t)\|^{1/2\gamma}_{\cal H_{-\gamma}}= 0.
\end{align}
Since $p_\gamma=p_{\alpha_0-\eta}$, by the  arbitrariness of $\eta$, formula \eqref{4.13} holds  for $1\leqslant p< p_{\alpha_0}$.

Therefore, by Lemma  \ref{L4.1},  the family of global attractors  $\{\mathcal{A}_\alpha\}$   is upper semicontinuous at the point $\alpha_0$, i.e.,  formula \eqref{Equ4.1} holds.
\end{proof}


\begin {thebibliography}{90} {\footnotesize

\bibitem{Abergel}
F. Abergel, Existence and finite dimensionality of the global attractor for evolution equations on unbounded domains,
J. Differential Equations, 83  (1990) 85-108.

\bibitem{Babin}
A. V. Babin, M. I. Vishik, Attractors of partial differential evolution equations in an unbounded domain,
Proceedings of the Royal Society of Edinburgh, 116A (1990) 221-243.

\bibitem{Belleri}
V. Belleri, V. Pata, Attractors for semilinear strongly damped wave equations on $\mathbb{R}^3$, Discrete Contin. Dyn. Syst. 7 (2001) 719-735.

\bibitem{Carvalho2}
A. N. Carvalho, J. W. Cholewa, Regularity of solutions on the global attractor for a semilinear damped wave equation, J. Math. Anal. Appl. 337 (2008) 932-948.

\bibitem{Chemin}
J. Y. Chemin, Localization in Fourier space and Navier-Stokes system, Phase space analysis of partial differential equations, Vol. I. CRM series, Pisa; Centro, Edizioni, Scunla Normale superiore 53-135, 2004.

\bibitem{S-Chen}
S. P. Chen, R. Triggiani, Proof of two conjectures of G. Chen and D.L. Russell on structural damping for elastic systems, Lecture Notes in Math., vol. 1354, Springer-Verlag, 1988, pp. 234-256.

\bibitem{S-Chen1}
S. P. Chen, R. Triggiani, Proof of extension of two conjectures on structural damping for elastic systems, Pacific J. Math. 136 (1989) 15-55.

\bibitem{Chueshov1}
I. Chueshov, Global attractors for a class of Kirchhoff wave models with a structural nonlinear damping, J. Abstr. Differ. Equ. Appl. 1 (2010) 86-106.

\bibitem{Chueshov}
I. Chueshov, Long-time dynamics of Kirchhoff wave models with strong nonlinear damping, J. Differential Equations, 252 (2012) 1229-1262.

\bibitem{Conti}
M. Conti, V. Pata, M. Squassina,  Strongly damped wave equations on $\mathbb{R}^3$ with critical nonlinearities, Comm. Appl. Anal. 9  (2005) 161-176.

\bibitem{Y-D-L1}
P. Y. Ding, Z. J. Yang, Y. N. Li, Global attractor of the Kirchhoff wave models with  strong nonlinear damping, Appl. Math. Lett.  76  (2018) 40-45.

\bibitem{D-Y}
P. Y. Ding, Z. J. Yang, Attractors of the strongly damped Kirchhoff wave equation on $R^N$,  Comm. Pure Appl. Anal.  18 (2) (2019)   825-843.

\bibitem{Feireisl1}
E. Feireisl, Attractors for semilinear damped wave equations on $\mathbb{R}^3$, Nonlinear Anal.  23 (1994) 187-195.

\bibitem{Feireisl2}
E. Feireisl, Asymptotic behaviour and attractors for a semilinear damped wave equation with supercritical exponent, Proceedings of the Royal Society of Edinburgh, 125A (1995) 1051-1062.

\bibitem{G-G-H}
M. Ghisi, M. Gobbino, A. Haraux, Local and global smoothing effects for some linear hyperbolic equations with a strong dissipation,Trans. Amer. Math. Soc. 368 (2016) 2039-2079.

\bibitem{Karachalios}
N. I. Karachalios, N. M. Stavrakakis, Existence of a global attractor for semilinear dissipative  wave equations on $\R^N$, J. Differential Equations, 157 (1999) 183-205.

\bibitem{Kenig}
C. E. Kenig, G. Ponce, L. Vega, Well-posedness and scattering results for the generalized Korteweg-de Vries equation via the contraction principle, Comm. Pure. Appl. Math. 46 (1993) 527-620.

\bibitem{M-S}
X. Y. Mei, C. Y. Sun, Attractors for a sup-cubic weakly damped wave equation in $\mathbb{R}^3$, Discrete Contin. Dyn. Syst. Doi:10.3934/dcdsb.2019053.

\bibitem{Mielke}
A. Mielke, G. Schneider,  Attractors for modulation equations on unbounded domains-existence and comparison, Nonlinearity, 8 (1995) 743-768.

\bibitem{S}
A. Savostianov, Infinite energy solutions for critical wave equation with fractional damping in unbounded domains,  Nonlinear Anal. 136   (2016) 136-167.

\bibitem{Simon}
J. Simon, Compact sets in the space $L^p(0,T;B)$, Ann. Mat. Pura Appl. 146 (1986) 65-96.

\bibitem{St}
W. A. Strauss, On continuity of functions with values in various Banach spaces, Pacific J. Math. 19 (1966) 543-551.

\bibitem{Wang}
B. X. Wang, Attractors for reaction-diffusion equations in unbounded domains, Phys. D 128 (1999) 41-52.

\bibitem{WWQ}
L. Z. Wang, Y. H. Wang, Y. M. Qin, Upper semicontinuity of attractors for nonclassical diffusion equations in $H^1(\mathbb{R}^3)$, Applied Mathematics and Computation, 240 (2014) 51-61.

\bibitem{Yang}
Z. J. Yang,  Longtime behavior of the Kirchhoff type equation with strong damping on $ \mathbb{R}^{N}$, J. Differential Equations,   242 (2007) 269-286.

\bibitem{Y-D}
Z. J. Yang, P. Y. Ding, Longtime dynamics of the Kirchhoff equation with  strong damping and critical nonlinearity on $\mathbb{R}^{N}$, J. Math. Anal. Appl. 434 (2016) 1826-1851.

\bibitem{Y-D-L}
Z. J. Yang, P. Y. Ding, L. Li, Longtime dynamics of the Kirchhoff equations with fractional damping and supercritical nonlinearity, J. Math. Anal. Appl. 442 (2016) 485-510.

\bibitem{Y-L}
Z. J. Yang, Y. N. Li, Upper semicontinuity of pullback attractors for non-autonomous Kirchhoff wave equations, Discrete Contin. Dyn. Syst.  Doi:10.3934/dcdsb.2019036.

\bibitem{Y-L-N}
Z. J. Yang, Z. M Liu, P. P. Niu, Exponential attractor for the wave equation with structural damping and supercritical exponent, Communications in Contemporary Mathematics, (2015) 1550055.

\bibitem{Zelik}
S. V. Zelik, Attractors of reaction-diffusion systems in unbounded domains and their spatial complexity, Comm. Pure. Appl. Math. 56  (2003) 584-637.

\bibitem{Zelik1}
S. V. Zelik, The attractor for a nonlinear hyperbolic equation in the unbounded domain, Discrete Contin. Dyn. Syst. 7 (2001) 593-641.

}
\end{thebibliography}
\end{document}